\documentclass[pdflatex,sn-mathphys-num, 12pt]{sn-jnl}

\usepackage{graphicx}
\usepackage{multirow}
\usepackage{amsmath,amssymb,amsfonts}
\usepackage{amsthm}
\usepackage{mathrsfs,caption}
\usepackage[title]{appendix}
\usepackage{xcolor}
\usepackage{manyfoot}
\usepackage{booktabs}
\usepackage{algorithm}
\usepackage{algorithmicx}
\usepackage{algpseudocode}
\usepackage{listings}
\usepackage{bm}
\usepackage{bbm}
\usepackage[Symbol]{upgreek}
\usepackage{soul}

\usepackage{caption}
\usepackage{subcaption}
\usepackage[T1]{fontenc}
\usepackage{needspace}
\usepackage{hyperref}
\usepackage{placeins}

\theoremstyle{thmstyleone}
\newtheorem{theorem}{Theorem}
 
\newtheorem{proposition}{Proposition}

\newtheorem{remark}{Remark}
\newtheorem{conjecture}{Conjecture}

\newtheorem{corollary}{Corollary}

\newcommand{\be}{\begin{equation}}
\newcommand{\ee}{\end{equation}}

\begin{document}

\begin{center}
{\Large \bf Geometric properties \vskip 3mm of the Lebesgue function}
\vskip 7mm
{\bf Leokadia Białas-Cież\footnote{\label{JU_Faculty}Faculty of Mathematics and Computer Science of the Jagiellonian University (Poland)}, Stefano De Marchi\footnote{Department of Medicine - University of Padova (Italy) and Padova Neuroscience Center - University of Padova} and Mateusz Suder\footref{JU_Faculty}}  
\end{center}

\vskip 10mm

\begin{center}
\small \textbf{Abstract}
\end{center}

\noindent
{\small
We present a collection of observations concerning the peculiar behavior of the Lebesgue function in the setting of the interval $[-1,1]\subset \mathbb{R}$ and the square $[-1,1]^2\subset \mathbb{R}^2$. We provide theoretical and numerical results and formulate
several open problems related to the geometry of the Lebesgue function.
}


\section{Introduction}

Let $N_n$  be the dimension of the space $\mathcal{P}_n(\mathbb{R}^N)$ of polynomials of $N$ variables and of degree at most $n$, i.e. $N_n=\binom{N+n}n$. For a compact set $E\subset \mathbb{R}^N$ with nonempty interior consider a family of points $\mathcal{X}_n=\{x_1,\dots,x_{N_n}\}\subset E$ 
that form a unisolvent set for $\mathcal{P}_n(\mathbb{R}^N)$, i.e. the determinant of the Vandermonde matrix 
$[x_j^\alpha]_{j=1,\dots,N_n, \, |\alpha|\le n}$ at these points does not vanish. As usual, $z^\alpha=z_1^{\alpha_1}\dots\,z_N^{\alpha_N}$ for $z\in \mathbb{R}^N$ and multiindex $\alpha$.
Let $\mathcal{C}(E)$ be the set of continuous functions on $E$. For any $f\in \mathcal{C}(E)$, let $L_n f$ be the interpolation polynomial of $f$ related to points from $\mathcal{X}_n$ given in the following form
\[ L_n f(x) \!=\!\sum_{j=1}^{N_n} f(x_j) \, \ell_j(x), \ \ \ \ \ell_j(x)\!=\!\ell_{x_j}^{\mathcal{X}_n}(x)\!=\!\frac{V(x_1,\dots,x_{j-1},x,x_{j+1},\dots,x_{N_n})}{V(x_1,\dots,x_{N_n})}, \ \ x\!\in\! \mathbb{R}^N\!,\]
where $V(t_1,\dots,t_{N_n})$ is the determinant of the Vandermonde matrix $[t_j^\alpha]_{j=1,\dots,N_n, \, |\alpha|\le n}$ and $\ell_j$ is called the \textit{$j$th fundamental Lagrange interpolation polynomial}. The \textit{Lebesgue function} is defined by
\begin{equation}\label{eq:lebfun} \lambda_n(x)= \lambda_n^{\mathcal{X}_n}(x):=\sum_{j=1}^{N_n} \left| \ell_j(x) \right|\end{equation}
and the \textit{Lebesgue constant} is its sup norm on $E$, i.e., 
\begin{equation} \label{eq:lebcons}
\Lambda_n=\Lambda_n(E,\mathcal{X}_n):=\max\{\lambda_n(x)\: : \: x\in E\} =\|\lambda_n\|_E.
\end{equation}

\noindent The Lebesgue function \eqref{eq:lebfun} and the Lebesgue constant \eqref{eq:lebcons} play a significant role in interpolation and approximation theory. Indeed, it is well known that $\Lambda_n$ is equal to the norm of the interpolation projection $L_n:\mathcal{C}(E) \ni f \mapsto L_nf \in \mathcal{P}_n(\mathbb{R}^N)$. Moreover, the Lebesgue inequality provides a bound on the interpolation error in terms of $\Lambda_n$. Additionally, the Lebesgue function $\lambda_n$ controls the stability of the operator $L_n$, due to the following estimate
$$ |L_n f(x) - L_n g(x)| \ \le  \ \lambda_n(x) \ \|f-g||_{\mathcal{X}_n}  $$
for any $x\in E$ and functions $f,g\in\mathcal{C}(E)$. Finally, the Lebesgue constant represents the conditioning of the interpolation problem, as
\begin{equation}
\|p_n-\tilde{p}_n\|_E \le \Lambda_n \|u-\tilde{u}\|_{X_n}\,,
\end{equation}
where $u=(u_1, \ldots, u_{N_n}), \; u_i=p_n(x_i)$, $p_n \in \mathcal{P}_n$, while $\tilde{u}, \; \tilde{u}_i=\tilde{p}_n(x_i)$ is a perturbation of $u$ and $\tilde{p}_n$ the corresponding interpolating polynomial.
\vskip 2mm 
\noindent

\vskip 2mm 
\noindent
The paper is organized as follows. 
In Section \textit{2}, we investigate problems concerning the geometry of the Lebesgue function for interpolation sets on an interval, while Section \textit{3} addresses the corresponding problems on a square in $\mathbb{R}^2$. Subsection 2.1 focuses on theoretical aspects, whereas Subsection 2.2 and Section \textit{3} present the results of numerical experiments. Finally, in Section \textit{4}, we formulate open theoretical questions arising from the preceding sections, which may serve as a starting point for future analytical investigations.

\vskip 2mm 
\noindent
In Subsection $\textit{2.1.}$ we study the problem of the maximal points of the Lebesgue function for interpolation nodes on the interval. We briefly review several known cases in which the set of maximal points is explicitly known. We provide a general bound on the separation of interpolation nodes from the endpoints of the interval under which the maximum of the Lebesgue function does not occur at the boundary.
\vskip 2mm 
\noindent
Subsection $\textit{2.2.}$ focuses on determining the intervals of convexity of the Lebesgue function for interpolation nodes on the interval. We discuss several general observations and present numerical experiments for two specific sets of interpolation nodes, namely the Chebyshev points of the first and second kinds.
\vskip 2mm 
\noindent
In Section $\textit{3}$, we consider the problem of local maxima of the two-dimensional Lebesgue function for the Padua points and the Morrow-Patterson points, and we discuss the geometry of the corresponding surface of the Lebesgue function. 
Based on numerical experiments, we propose the lower bound for the number of local maxima and present examples for which the Lebesgue function has a larger number of local maxima.

\vskip 2mm
\noindent
Section \textit{4.} is devoted to concluding remarks and open problems. We formulate three questions for further research. They concern the distribution of interpolation nodes with regard to the location of global maxima of the Lebesgue function, and its convexity and concavity properties. In addition, we ask about the number of local maxima of the Lebesgue function on the square.

\section{One-dimensional case}

In this section, we examine interpolation problems on an interval. Subsection 2.1 deals with theoretical aspects, while Subsection 2.2 reports on the numerical experiments.

\vskip 2mm
\noindent
Let $I:=[-1,1]$ and $\mathcal{X}_n$ be a family  of strictly increasing nodes in $I$: $-1\le x_1<\dots<x_{n+1}\le 1$.  It is well known that in this case the interpolation fundamental Lagrange polynomials can be written in the form
$$ \ell_j(x)= \prod_{k=1,\, k\ne j}^{n+1} \frac{x-x_k}{x_j-x_k}, \ \ \ j=1,\dots,n+1, \ x\in I. $$ 
Consider the Lebesgue function $\lambda_n=\sum_{j=1}^{n+1} |\ell_j|$ related to $\mathcal{X}_n$. The minimal value of $\lambda_n$ in $I$ is equal to 1 and is attained only at points of $\mathcal{X}_n$ for $n>1$. Therefore, from now on, we assume that $n\in\{2,3,\dots\}$.

\subsection{Maximum points of the Lebesgue function}

If $x_{n+1}<1$ then $|\ell_j|$ for every $j$ is strictly increasing and convex in $(x_{n+1},1)$ and thus the Lebesgue function shares the same properties there. Similarly, $\lambda_n$ is strictly decreasing and convex in $[-1,x_1]$ for $x_1>-1$. 

\vskip 2mm
\noindent 
Moreover, the Lebesgue function is piecewise polynomial. To be more precise, consider an interval $(x_k,x_{k+1})\subset I$ for $k=1,\dots,n$. 
Then there exists a polynomial $p_k\in\mathcal{P}_n$ such that $p_k=\lambda_n$ in $[x_k,x_{k+1}]$. 
By analyzing the zeros of the derivative of $p_k$, one can show that $\lambda_n$ attains its maximum on $[x_k,x_{k+1}]$ at exactly one point, see \cite{LR65}. Thus, at the node $x_k$, the Lebesgue function takes the value 1, then it increases until a local maximum, and afterward it decreases back to 1 at $x_{k+1}$.

\vskip 2mm
\noindent 
We are interested in such points in $I$ where the Lebesgue function  $\lambda_n$ attains its maximum on $I$. 
Let 
\[ A=A(\mathcal{X}_n):=\{x\in I\: : \: \lambda_n(x)=\Lambda_n\}. \] Observe that $A$ is a discrete set and its cardinality does not exceed $n+2$. Moreover, $\mathcal{X}_n \cap A=\emptyset$, $n\ge 2$.
The structure of the set $A$ seems very difficult to study. Even a slight perturbation of the nodes may completely disrupt it, see Fig. \ref{Cheb_small_perturbation}. The classical cases, now fairly well understood, required extensive and laborious investigation.

\begin{figure}[htbp]
     \begin{subfigure}[t]{0.47\textwidth}
        \includegraphics[width=\textwidth]{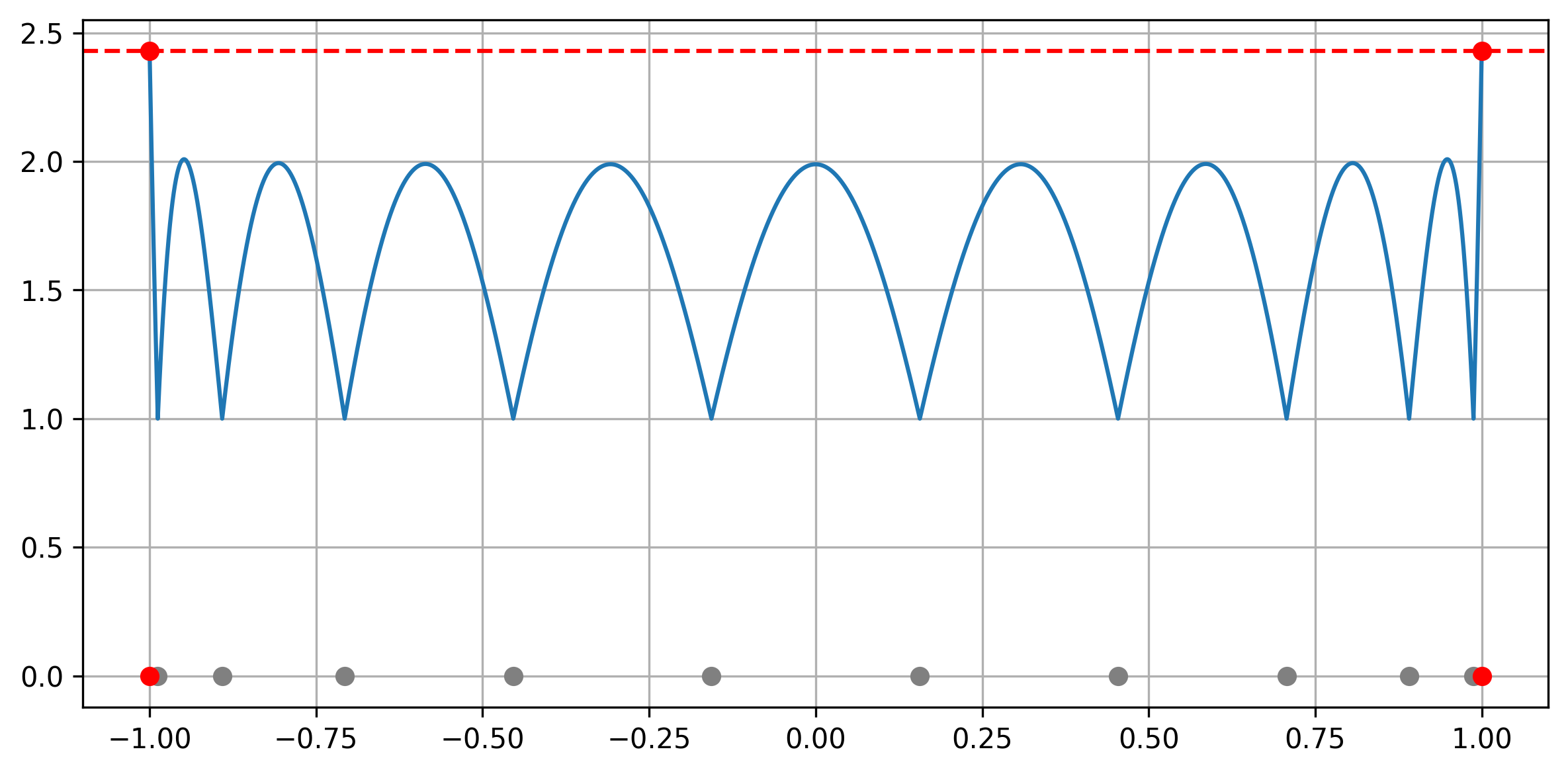}
        \caption{The Lebesgue function for $10$ Chebyshev points of the first kind. The grey points indicate the interpolation nodes.}
    \end{subfigure}
    \hspace{.5cm}
    \begin{subfigure}[t]{0.47\textwidth}
        \includegraphics[width=\textwidth]{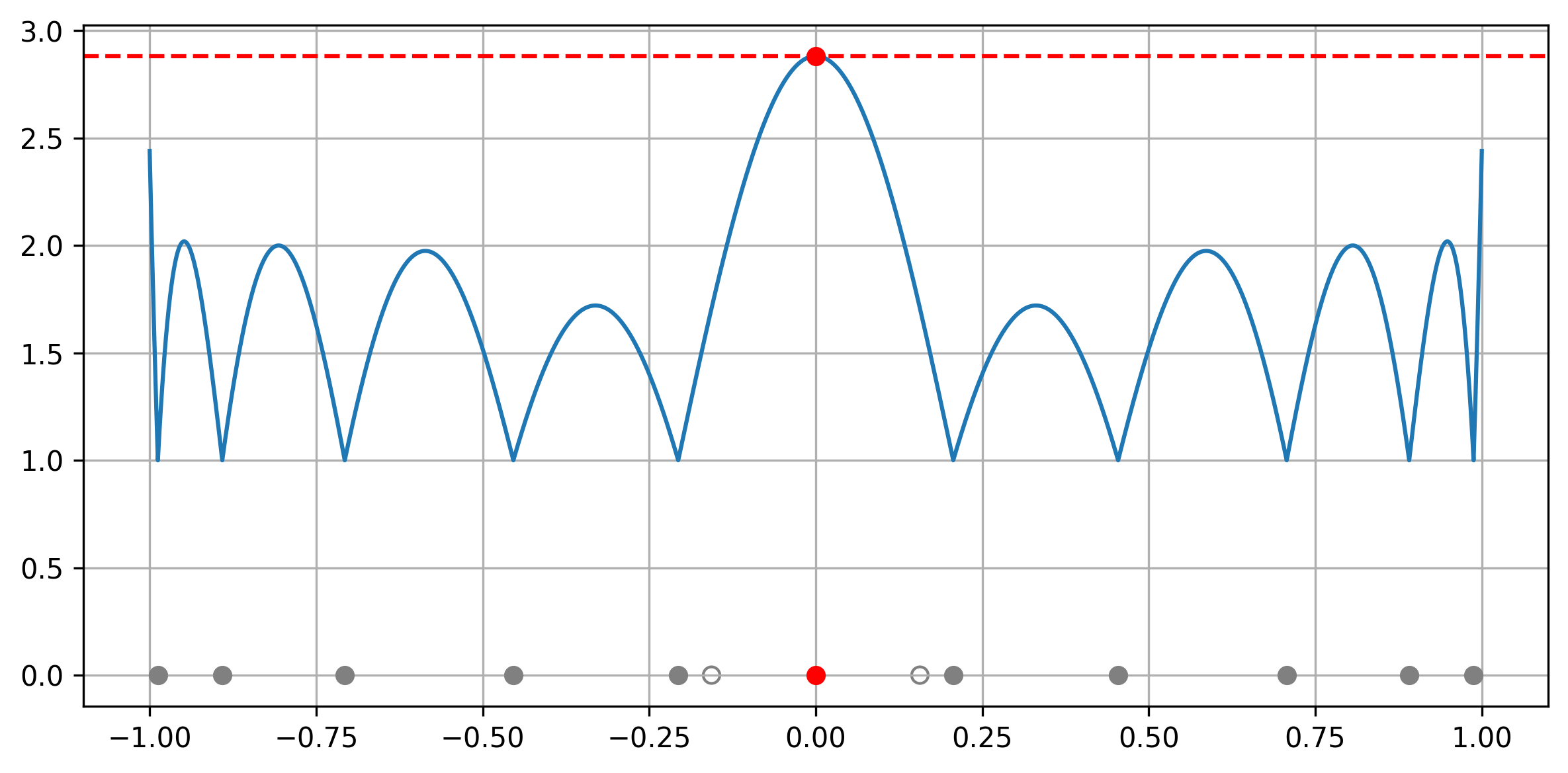}
        \caption{The Lebesgue function after moving each of the two central nodes outward by $0.05$. Filled grey points indicate the modified nodes, while hollow grey points mark their original positions.}
    \end{subfigure}
    
     \begin{subfigure}[t]{0.47\textwidth}
        \includegraphics[width=\textwidth]{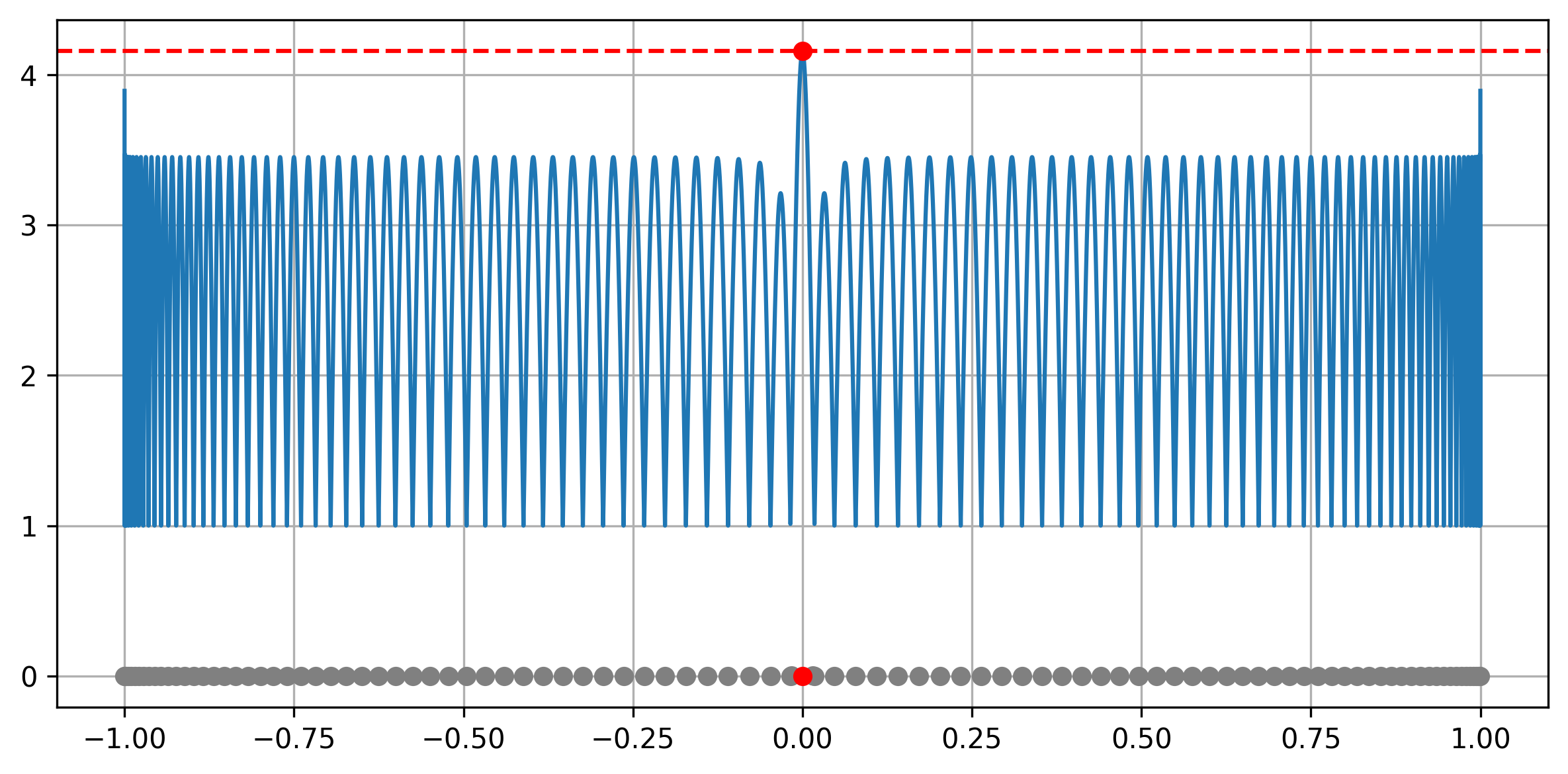}
        \caption{The Lebesgue function for $100$ Chebyshev points of the first kind, with each of the two central nodes moved outward by $0.002$.}
    \end{subfigure}
    \hspace{.5cm}
    \begin{subfigure}[t]{0.47\textwidth}
        \includegraphics[width=\textwidth]{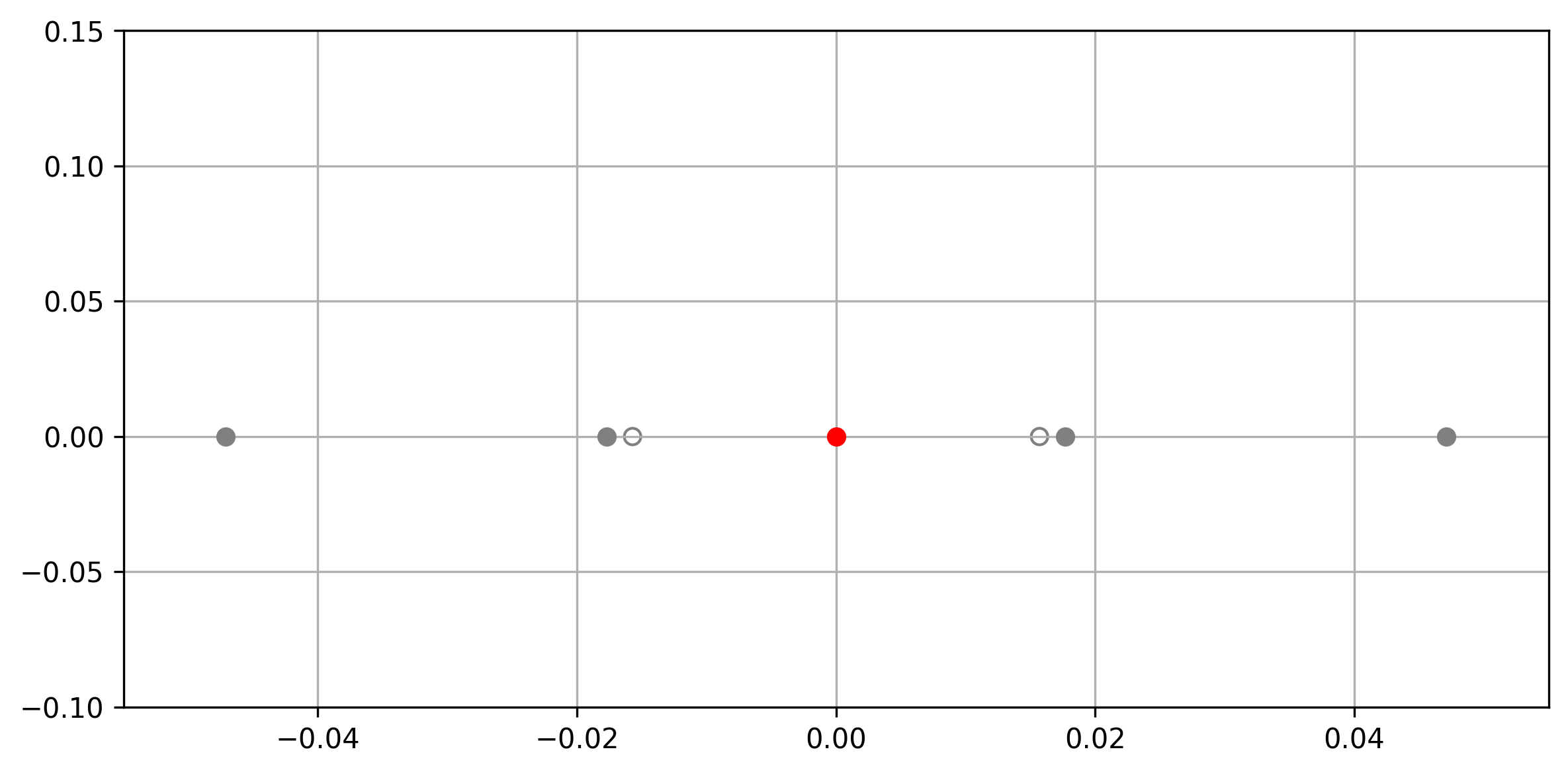}
        \caption{Close-up view of the interpolation nodes from (c). Hollow gray points are the original positions of the moved nodes.}
    \end{subfigure}
    \caption{Enlargement showing the central interpolation nodes from panel~(c). Filled grey points indicate the modified nodes, while hollow grey points mark their original positions.}
    \label{Cheb_small_perturbation}
\end{figure}

\noindent
For the equidistant nodes $x_j=-1+\tfrac{2(j-1)}{n}, \ j=1,\dots,n+1$, the set $A$ consists of two symmetric points with respect to zero: one of them is contained in $(x_1,x_2)$, and the other in its symmetric counterpart $(x_n,x_{n+1})=(-x_2,-x_1)$. Moreover, $\Lambda_n\sim \tfrac{2^{n+1}}{en\, \log n}$, as $n\rightarrow \infty$. 

\vskip 2mm
\noindent
The Lebesgue function of the Chebyshev nodes (of the first kind) $x_j=-\cos \tfrac{(2j-1)\pi}{2n+2}$, $j=1,\dots,n+1$, attains its maximal value precisely at -1 and 1, so $A=\{-1,1\}$, as was shown in 1965 by Powell. A classical Bernstein result states that $\Lambda_n\sim \tfrac{2}{\pi}\log (n+1)$, $n\rightarrow \infty$. 

\vskip 2mm
\noindent
The set $A$ for the Chebyshev-Lobatto nodes consisting of extremal points of the Chebyshev polynomial of degree $n$, i.e. $x_j=-\cos \tfrac{(j-1)\pi}{n}$, $j=1,\dots,n+1$, depends on the parity of $n$. Specifically, if $n$ is odd, then $A=\{0\}$, and $A$ contains two symmetric points close to zero if $n$ is even. However, numerical experiments show that $-1\in A$ after removing this point from the set of nodes. The Lebesgue constant for the Chebyshev–Lobatto nodes has the same asymptotic growth as that for the Chebyshev points.

\vskip 2mm
\noindent
The Chebyshev nodes of the second kind: $x_j=-\cos \tfrac{j\pi}{n+2}$, $j=1,\dots,n+1$ lie in the interior of $I$ and $A=\{-1,1\}$. For these points, the exact value of the Lebesgue constant is known to be $\Lambda_n=n+1$.

\vskip 2mm
\noindent
The extended Chebyshev nodes $x_j=-\cos \tfrac{(2j-1)\pi}{2n+2} \cdot (\cos \tfrac\pi{2n+2})^{-1}$, $j=1,\dots,n+1$ appear to have a nearly optimal Lebesgue constant \cite{Ibrahimoglu2016}. Namely, their $\Lambda_n$ satisfies inequalities (\ref{optimal}), see below. The endpoints of $I$ belong to this family of nodes, and the set $A$ contains one or two points at or near zero (respectively), depending on the parity of $n$. 

\vskip 2mm
\noindent
Bernstein already observed that there exists a set of nodes $\mathcal{X}_n^\star\subset I$ satisfying the optimality condition: $$\min_{\mathcal{X}_n\subset I} \Lambda_n(I,\mathcal{X}_n) = \Lambda_n(I,\mathcal{X}_n^\star)=: \Lambda_n^\star.$$ 
In 1978 Kilgore \cite{K78} and de Boor, Pinkus \cite{BP78} proved that if $-1,1\in \mathcal{X}_n$ and the cardinality of $A(\mathcal{X}_n)$ is maximal, i.e. card$\, A(\mathcal{X}_n)=n$, then $\mathcal{X}_n=\mathcal{X}_n^\star$. This result resolves the long-standing Bernstein-Erdős conjecture. A nice estimate for $\mathcal{X}_n^\star$ was proved by Brutman in \cite{B78}
\begin{equation}\label{optimal}
0.5212 +\tfrac2\pi \log(n+1) < \Lambda_n^\star < 0.75 +\tfrac2\pi \log(n+1)\ \ \ \ \mbox{for all} \ n=2,3,\dots
\end{equation}
The number $0.5212$ is a close approximation of the quantity $\tfrac2\pi(\gamma+\log\tfrac4\pi)$, where $\gamma$ denotes the Euler constant. 

\vskip 2mm
\noindent
If necessary, taking a subinterval $I'\subset I$ and an affine transformation between $I'$ and $I$, we see that $\Lambda^\star_n\le \Lambda_n(I,\mathcal{X}_n)$ for any family of nodes $\mathcal{X}_n\subset I$. Hence the assumption $-1,1\in \mathcal{X}_n$ is not needed for the left-hand side inequality in (\ref{optimal}), and so  
\begin{equation}\label{eq:lebesgue_bound} 
0.5212 +\tfrac2\pi \log(n+1) \ < \ \Lambda_n(\mathcal{X}_n) \ \ \ \ \mbox{for all} \  n=2,3,\dots \ \ \mbox{and any } \mathcal{X}_n.
\end{equation}
Very recently, Terence Tao proved in \cite[Thm. 1.10]{TT} that, for every fixed interval $[a,b]\subset I$ with $a<b$, there exist a constant $C_o>0$ depending only on $a, b$, such that
$$ \| \lambda_n \|_{[a,b]} > \tfrac2\pi \log (n+1) - C_o $$ 
for any $\mathcal{X}_n\subset I$, thereby solving Erdős-Tur\'an Problem 1153 in \cite{Bloom}. 

\vskip 2mm
\noindent
However, the right-hand inequality in (\ref{optimal}) regards optimal family of nodes $\mathcal{X}_n^*$ containing the endpoints of $I$. We can slightly rescale $\mathcal{X}_n^\star$ to obtain a new set of nodes $\mathcal{X}_n^{\star\star}$ contained in the interior of $I$, for which the same estimate still holds. But what does ‘slightly’ mean? In the following, we provide a sufficient geometric condition on a family of nodes $\mathcal{X}_n\subset (-1,1)$ ensuring that $A(\mathcal{X}_n)\subset(-1,1)$. In other words, we are interested in such nodes that their Lebesgue function attains its maximum in the interior of $I$.

\vskip 4mm
\begin{proposition}
    Let $$ a(n):= 1 - \tfrac1{\gamma_n}
    \quad  \ \ \  \mbox{for} \ \ n=2,3,\dots,$$
where 
$$ \gamma_n:= \tfrac2\pi\log(n+1) +0.5212. $$
    Then 
    $$ 1- \tfrac{a(n)}{n^2} \le x_{n+1} \ \implies \ 
     1\not\in A(\mathcal{X}_n), $$ \vskip 1mm
    $$ x_1\le -\left( 1- \tfrac{a(n)}{n^2} \right) \ \implies \ 
    -1\not\in A(\mathcal{X}_n). $$ 
\end{proposition}

\begin{proof}
    We restrict ourselves to $x_1$. The case $x_1=-1$ is obvious, so we may assume that $x_1>-1$.
    Consider such a polynomial $p\in\mathcal{P}_n$ that $\lambda_n = p$ on the interval $[-1,x_1]$.
By the classical Markov inequality, it follows that
$$ p(-1)-p(x_1) \le (x_1+1) \|p'\|_{I} \le (x_1+1)\, n^2\, \|p\|_{I}. $$
Since $p(-1)=\lambda_n(-1)$ and  $p(x_1)=1$, we get 
\be \label{techn_1} \lambda_n(-1) \le 1 + a(n)\, \|p\|_I , \ee 
providing that $x_1\le -\left( 1- \tfrac{a(n)}{n^2} \right)$.
Let $c_j=\tfrac{|\ell_j(-1)|}{\ell_j(-1)}$ for $j=1,\dots,n+1$. Then for $x\in[-1,x_1]$ we have 
$$ \lambda_n(x) = \sum_{j=1}^{n+1} |\ell_j(x)| = \sum_{j=1}^{n+1} c_j\, \ell_j(x) = p(x). $$
For an arbitrary interval $(x_k,x_{k+1})$, $k\in\{1,\dots,n\}$, or $(x_{n+1},1)$ (if it is not empty), all Lagrange fundamental polynomials $\ell_j$ are of constant signs, and consequently, 
$$ \lambda_n = \sum_{j=1}^{n+1} |\ell_j| = \max_{\alpha_1,\dots,\alpha_{n+1}\in[-1,1]} \sum_{j=1}^{n+1} \alpha_j\,\ell_j \ge \sum_{j=1}^{n+1} c_j\,\ell_j = p. $$
Similarly, $\lambda_n\ge -p$, and thus $|p|\le \lambda_n $ in every interval $(x_k, x_{k+1})$, $(-1,x_1)$, $(x_{n+1},1)$, so $\|p\|_I\le \|\lambda_n\|_I = \Lambda_n$. Applying this estimate and the definition of $a(n)$ to (\ref{techn_1}), we obtain
$$ \lambda_n(-1) \le 1+ \left(1-\tfrac1{\gamma_n}\right) \,\Lambda_n  = \Lambda_n + 1 - \tfrac{\Lambda_n}{\gamma_n} < \Lambda_n, $$
the strict inequality being a consequence of (\ref{eq:lebesgue_bound}) that holds for any family $\mathcal{X}_n$. \\
The right endpoint of $I$ is treated identically.
\end{proof}

\noindent
Observe that the function $a=a(x)$ is strictly increasing and tends to 1 as $x\rightarrow \infty$. 
We can verify that $a(2)>0.18$, $a(6)>0.43$, $a(10)>0.51$ to obtain concrete examples. 
The inverse function of $x\mapsto a(x)$ is 
$$ N(b):=\exp\left[ \tfrac\pi2 \left( \tfrac1{1-b} -0.5212\right)\right] -1, \ \ \ a(2)\le b< 1.  $$
Since the range of $a(n)$ for $n>1$ is $[a(2),1)\approx [0.18,1) $
we can formulate the statement of the proposition in the following way.

\vskip 4mm 
\begin{theorem}
For any $b\in [a(2),1)$ and $n\ge N(b)$ 
   $$  x_1 \le - \left(1-\tfrac{b}{n^2} \right) \ \implies \ \lambda_n(-1) <\Lambda_n,$$
   $$  1-\tfrac{b}{n^2} \le x_{n+1} \ \implies \ \lambda_n(1) <\Lambda_n.$$
   \vskip 2mm
   \noindent
In particular, $\lambda_n(1) <\Lambda_n$ in each of the example cases below \begin{itemize}
        \item $ 1 - \tfrac1{2n^2} \le x_{n+1} $ and $n>9$, \vskip 2mm
        \item $ 1 - \tfrac2{5n^2} \le x_{n+1} $ and $n>5$, \vskip 2mm
        \item $ 1 - \tfrac{0.18}{n^2} \le x_{n+1} $ and $n>1$.
    \end{itemize}
\end{theorem}  
\noindent Indeed, in all cases, it is easy to check that $a(n) \ge b$ for $n \ge N(b)$.

\vskip 4mm
\begin{remark}
    {\rm Note that Chebyshev nodes $\mathcal{X}_n$ of the first and of the second kind have extreme points at a distance of order $\mathcal{O}(\tfrac1{n^2})$ from the endpoints of the interval $I$, but the maximum value of their Lebesgue function is attained at -1 and 1. However, this only appears to contradict Thm. 1, since a closer examination of the coefficients of $\tfrac1{n^2}$ for these classical nodes shows that the sufficient condition is not satisfied. Namely, for the Chebyshev nodes of the first kind, we have
    $$ 1+x_1=1-\cos\tfrac{\pi}{2n+2} \sim \tfrac{\pi^2}8 \, \tfrac1{n^2} \approx 1.2 \, \tfrac1{n^2},$$
    and for the second kind:
    $$ 1+x_1=1-\cos\tfrac{\pi}{n+2} \sim \tfrac{\pi^2}2 \, \tfrac1{n^2} \approx 4.9 \, \tfrac1{n^2}.$$
    The coefficients larger than 1 
    show that the hypothesis of Thm. 1 is not satisfied, so the theorem gives no information; this is consistent with the known fact that the maximum is attained at ±1.
    }
\end{remark}

\vskip 4mm
\begin{corollary}
    Scaling the set $\mathcal{X}_n$ that contains the endpoints of $I$, by a constant $c_n\in [1-\tfrac{a(n)}{n^2},1)$ for $n>1$, yields a new family of nodes $$\mathcal{X}'_n :=c_n \mathcal{X}_n \subset[-c_n, c_n]\subset(-1,1)$$ such that 
$$\Lambda_n(I,\mathcal{X}'_n) = \Lambda_n(I,\mathcal{X}_n).$$ 
In particular, by means of extended Chebyshev points, we can obtain a concrete family of nodes separated from the endpoints of $I$, with a nearly optimal Lebesgue constant satisfying inequalities (\ref{optimal}).
\end{corollary}

\begin{proof}
By the scaling property of the Lebesgue function, we have
\[ \lambda_{X'_n}(x)=\lambda_{X_n}(\tfrac{x}{c_n}), \]
and hence
\[ \max_{x\in[-c_n,c_n]}\lambda_{X'_n}(x)
=\|\lambda_{X_n}\|_{[-1,1]} . \]
Moreover, on each of the intervals \([-c_n,-1]\) and \([1,c_n]\), the function \(\lambda_{X'_n}\) is monotone towards the corresponding endpoint of $I$. From Prop. 1, neither 1 nor -1 belongs to $ A(\mathcal{X}_n')$, and consequently,
\[
\Lambda_n(I,\mathcal{X}'_n) = \|\lambda_{X'_n}\|_{I}
=\|\lambda_{X_n}\|_{I} = \Lambda_n(I,\mathcal{X}_n) ,
\]
which completes the proof.        
\end{proof}

\vskip 5mm

\subsection{Piecewise concavity and convexity of the Lebesgue function}

In 1974, Neuman stated that the Lebesgue function $\lambda_n$ is strictly concave in each interval $(x_k,x_{k+1})$, $k=1,\dots,n$; see \cite{N74}. The conjecture seems to be true when examining the graphs of the Lebesgue function for the classical interpolation nodes, see Fig.1 a) and Fig.2 a). In any interval between consecutive Chebyshev nodes, $\lambda_n$ is concave in a segment close to the point where the local maximum is reached, as was proved in \cite{B78}. Brutman also gave a counterexample in 1978 showing that $\lambda_4$ related to nodes $\{-1,-\alpha, 0, \alpha,1\}$ with $\alpha<0.45$ is convex near $\alpha$ in $(\alpha,1)$. But these nodes are far from optimal, because the Lebesgue constant for these nodes is greater than 2.66, e.g. $\Lambda_4\approx 3.30$ for $\alpha=0.4$, whereas for the Chebyshev nodes of the first kind $\Lambda_4\approx 2$.

\vskip 2mm 
\noindent
However, numerical experiments show that, even for the classical sets of interpolation points, such as Chebyshev nodes of the first and second kind, the Lebesgue function is convex on certain subintervals of $(x_n,x_{n+1})$ and their symmetric counterparts for sufficiently large $n$. 

\noindent Numerical experiments suggest that this behaviour extends to other subintervals as well, see Table \hyperref[Table_convexity_CheII]{1} and  Table \hyperref[Table_convexity_CheI]{2}, where we state the lowest degree of the Lebesgue function for which each of the $m$ intervals $(x_k, x_{k+1})$, $k=1,\dots,m$ contains a convex subinterval near the right endpoint $x_{k+1}$ (the left endpoint for their symmetric counterparts). More precisely, for every $k = 1,\dots,m$ there exists  $\varepsilon_{k,n}>0$ such that the restriction of the Lebesgue function to $(x_k, x_{k+1})$ is convex on $(x_{k+1} -\varepsilon_{k,n},\, x_{k+1})$, where $\varepsilon_{k,n}$ depends on both the index $k$ and the degree $n$. The provided numerical results are restricted to $m=1,\dots,20$ for the Chebyshev points of the second kind and $m=1,2,3$ for the Chebyshev points of the first kind.

\begin{figure}[ht]
\centering
    \begin{subfigure}[t]{0.48\textwidth}
        \includegraphics[width=\textwidth]{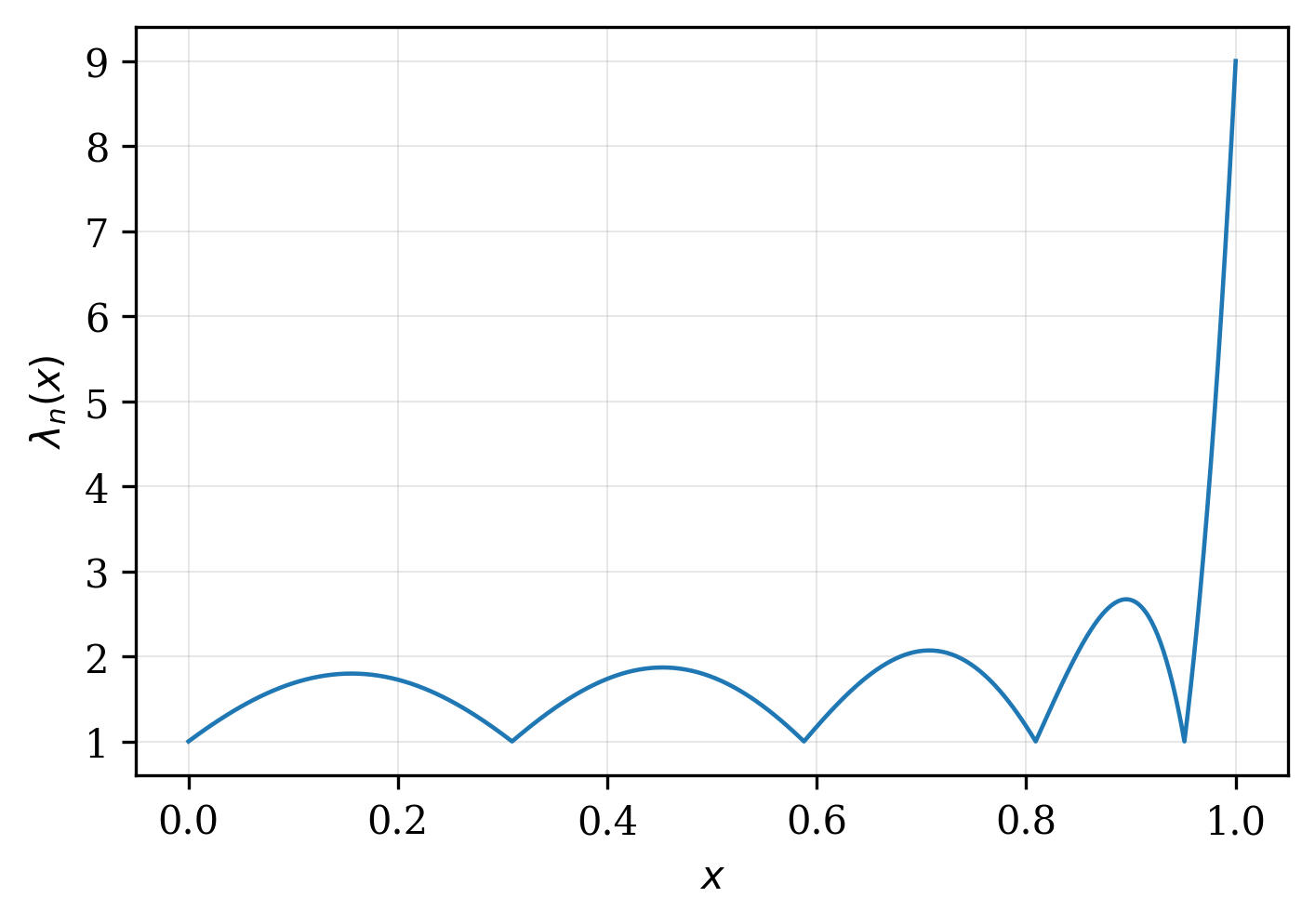}
    \caption{The Lebesgue function for Chebyshev points of the second kind with $n=8$, restricted to $[0,1]$.}
    \end{subfigure}
    \hspace{0.02\textwidth}
    \begin{subfigure}[t]{0.48\textwidth}
        \includegraphics[width=\textwidth]{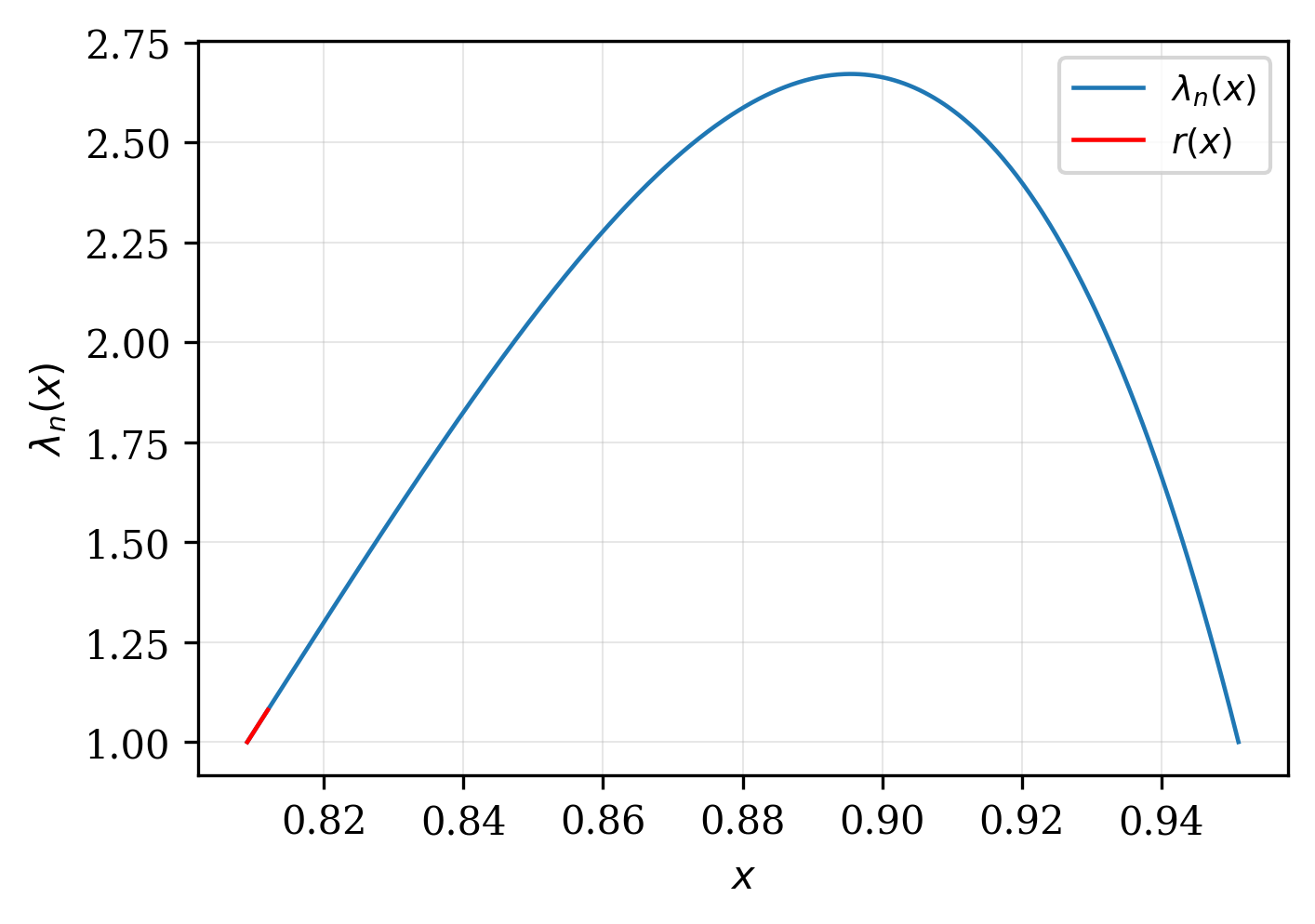}
    \caption{Enlargement of the interval $[x_n,x_{n+1}]$. The red segment joins two nearby points on the graph near $x_n$.}
    \end{subfigure}
    \begin{subfigure}[t]{0.48\textwidth}
        \includegraphics[width=\textwidth]{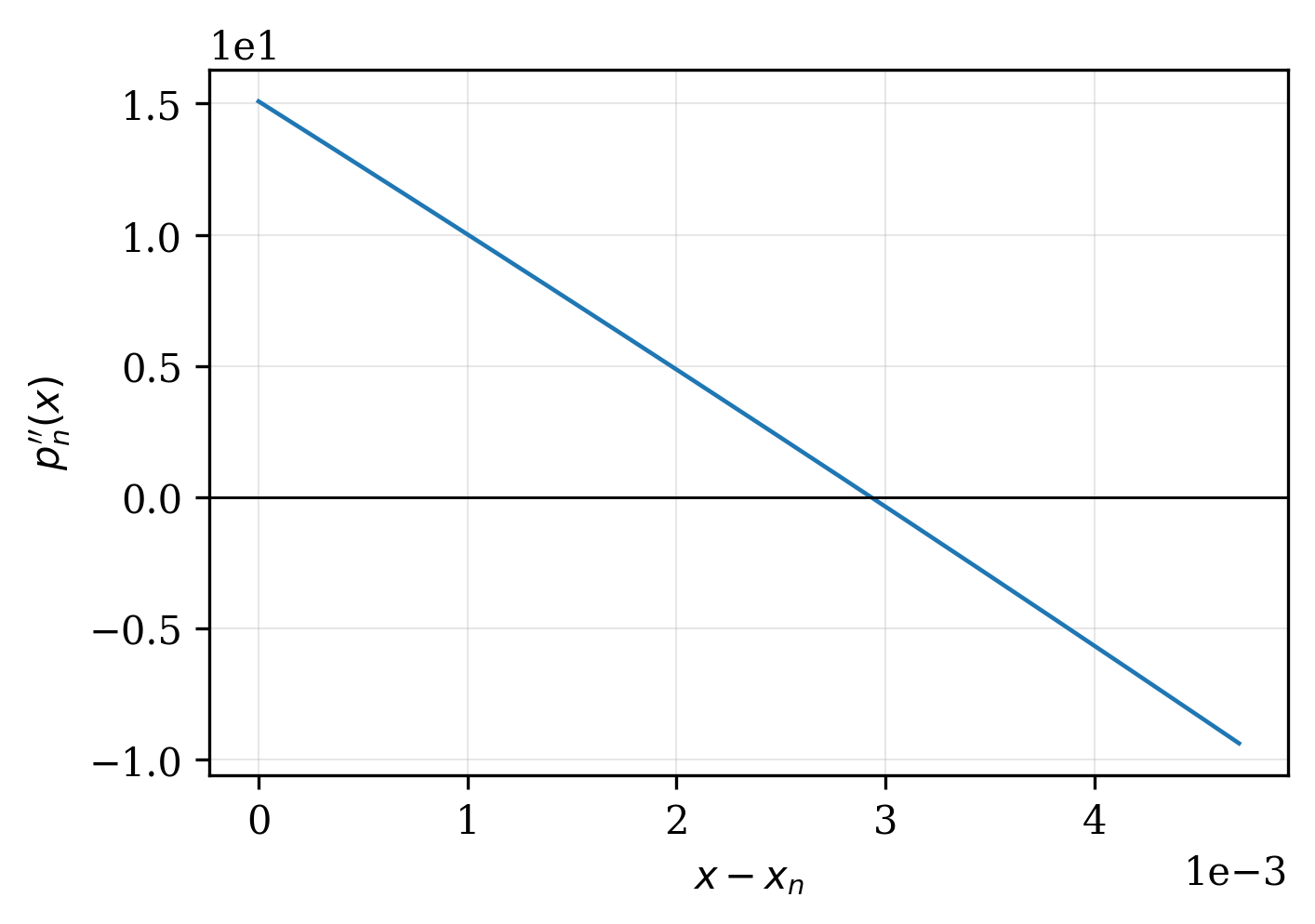}
    \caption{The second derivative of the Lebesgue function in a one-sided neighbourhood of the interpolation node $x_n$ within $(x_n,x_{n+1})$.}
    \end{subfigure}
    \hspace{0.02\textwidth}
    \begin{subfigure}[t]{0.48\textwidth}
    \includegraphics[width=\textwidth]{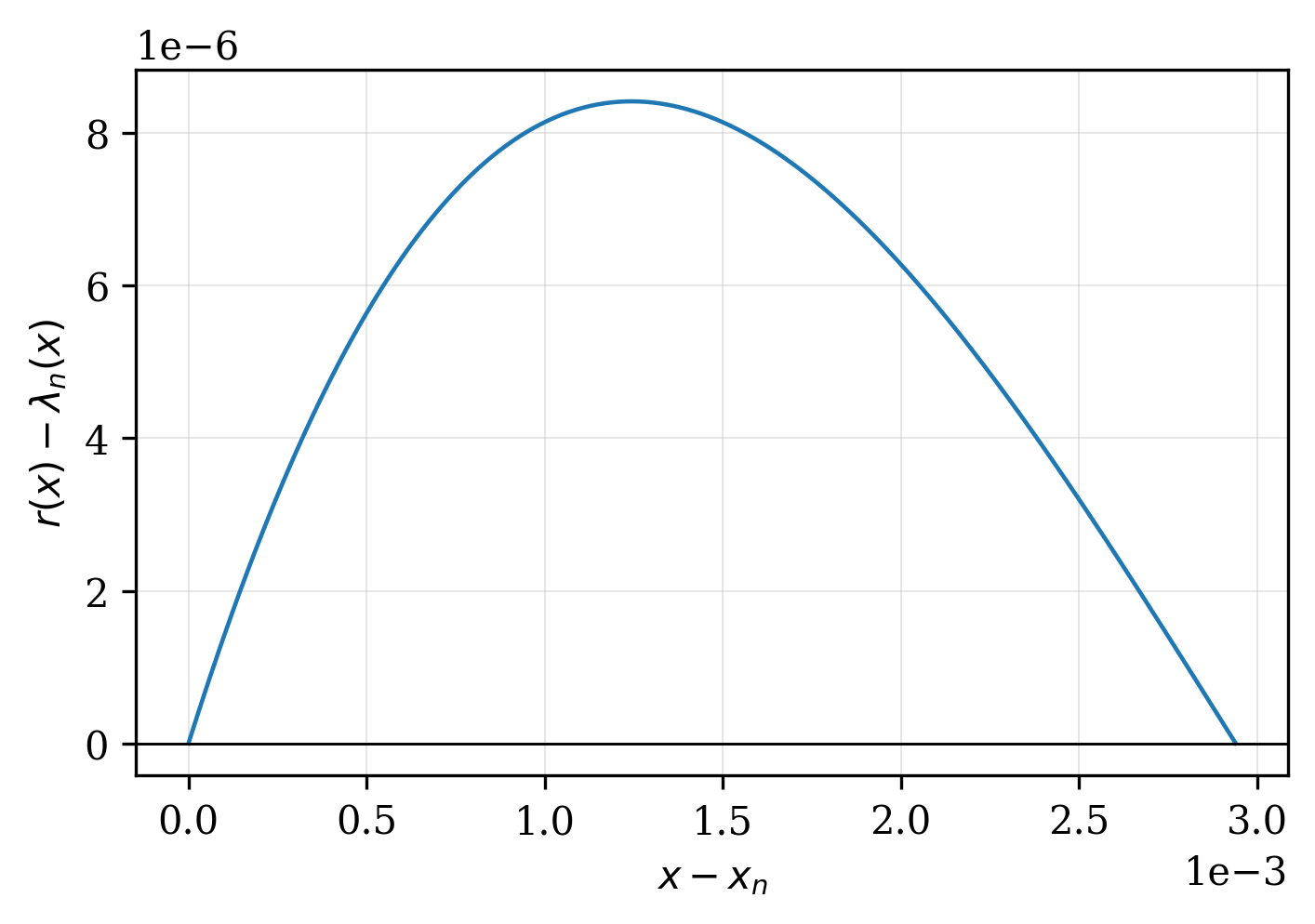}
    \caption{The difference $r(x)-\lambda_n(x)$ near $x_n$, where $r$ is the affine function joining the two selected points on the graph of $\lambda_n$.}
    \end{subfigure}
    \caption{The difference $r(x)-\lambda_n(x)$ near $x_n$, where $r$ is the affine function interpolating $\lambda_n$ at $x_n$ and $x_n+0.00294$.}
    \label{Convexity_interval_plot_ChebI}
\end{figure}
\begin{table}[ht]
\centering

\begin{tabular}{|c|c|c|c|c|}
\cline{1-2} \cline{4-5} 
Index $m$ & Min. degree $n$ & & Index $m$ & Min. degree $n$ \\ \cline{1-2} \cline{4-5} 
1  &  8   & & 11 & 181 \\ \cline{1-2} \cline{4-5}
2  & 16   & & 12 & 210 \\ \cline{1-2} \cline{4-5}
3  & 26   & & 13 & 241 \\ \cline{1-2} \cline{4-5}
4  & 38   & & 14 & 274 \\ \cline{1-2} \cline{4-5}
5  & 52   & & 15 & 309 \\ \cline{1-2} \cline{4-5}
6  & 68   & & 16 & 347 \\ \cline{1-2} \cline{4-5}
7  & 86   & & 17 & 386  \\ \cline{1-2} \cline{4-5} 
8  & 107   & & 18 & 428 \\ \cline{1-2} \cline{4-5}
9  & 129   & & 19 & 472 \\ \cline{1-2} \cline{4-5}
10  & 154   & & 20 & 518  \\ \cline{1-2} \cline{4-5} 
\end{tabular}
\captionsetup{width=\textwidth}
\caption{Minimal degree $n$, numerically detected, ensuring local convexity near $x_{k+1}$ on $(x_k,x_{k+1})$ for every $k=1,\dots,m$ for Chebyshev points of the second kind for $m=1,\dots,20$.}
\label{Table_convexity_CheII}
\end{table}

\begin{figure}[ht]
\centering
    \begin{subfigure}[t]{0.48\textwidth}
        \includegraphics[width=\textwidth]{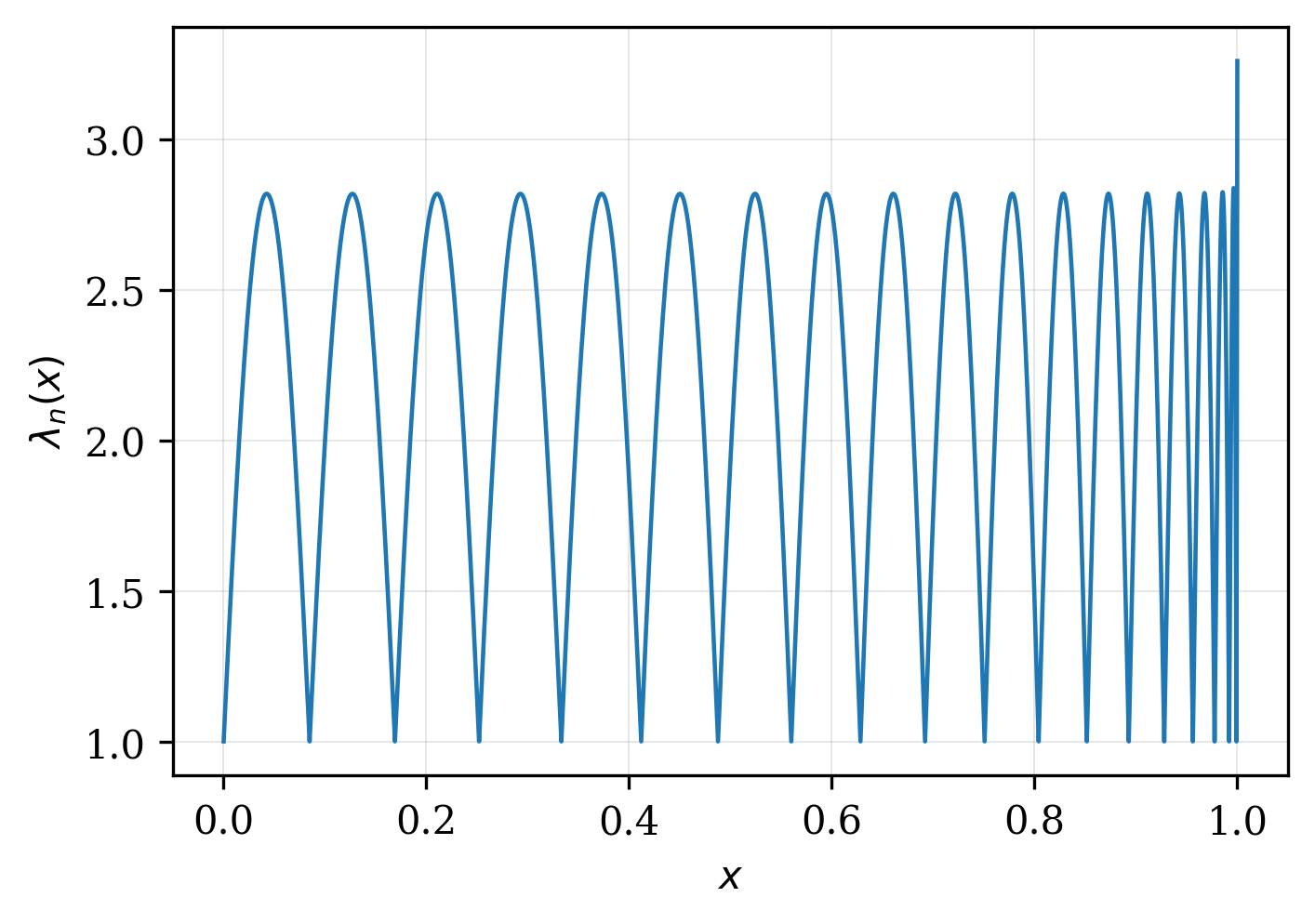}
    \caption{The Lebesgue function for Chebyshev points of the first kind with $n=36$, restricted to $[0,1]$.}
    \end{subfigure}
    \hspace{0.02\textwidth}
    \begin{subfigure}[t]{0.48\textwidth}    
        \includegraphics[width=\textwidth]{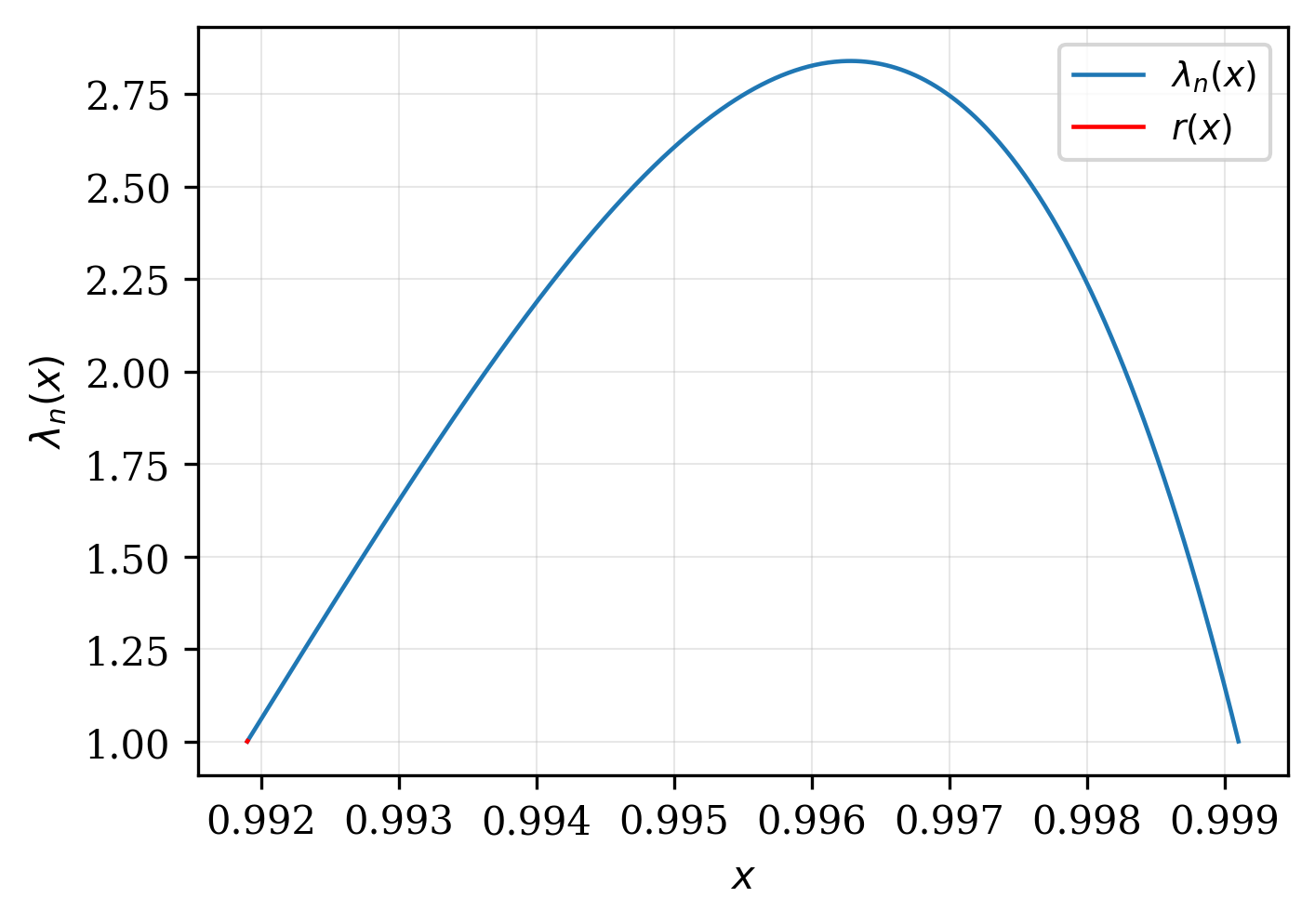}
    \caption{Enlargement of the interval $[x_{n},x_{n+1}]$. The red segment joins two nearby points on the graph near $x_n$.}
    \end{subfigure}
    \begin{subfigure}[t]{0.48\textwidth}
        \includegraphics[width=\textwidth]{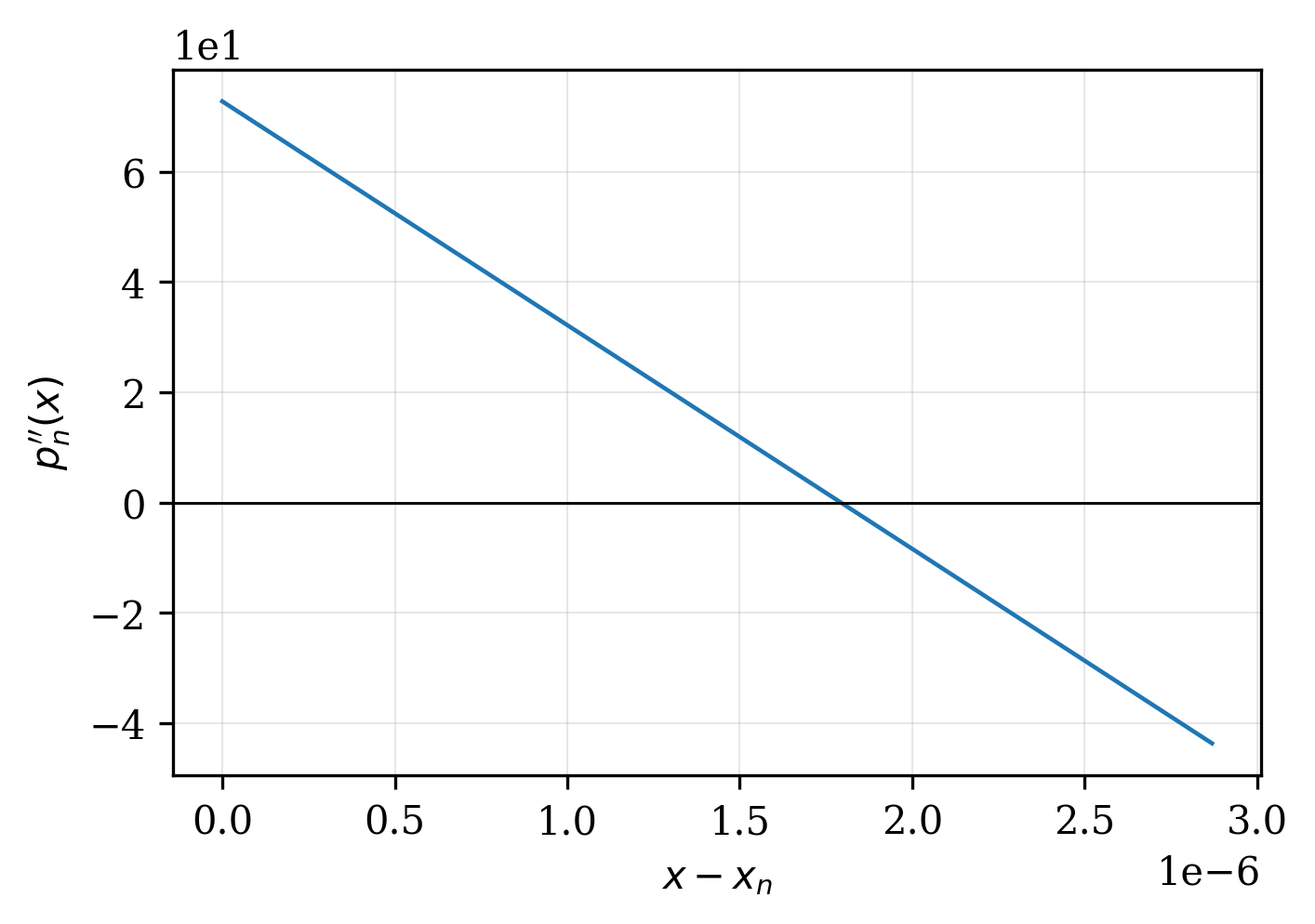}
    \caption{The second derivative of the Lebesgue function in a one-sided neighbourhood of the interpolation node $x_n$ within $(x_n,x_{n+1})$.}
    \end{subfigure}
    \hspace{0.02\textwidth}
    \begin{subfigure}[t]{0.48\textwidth}
    \includegraphics[width=\textwidth]{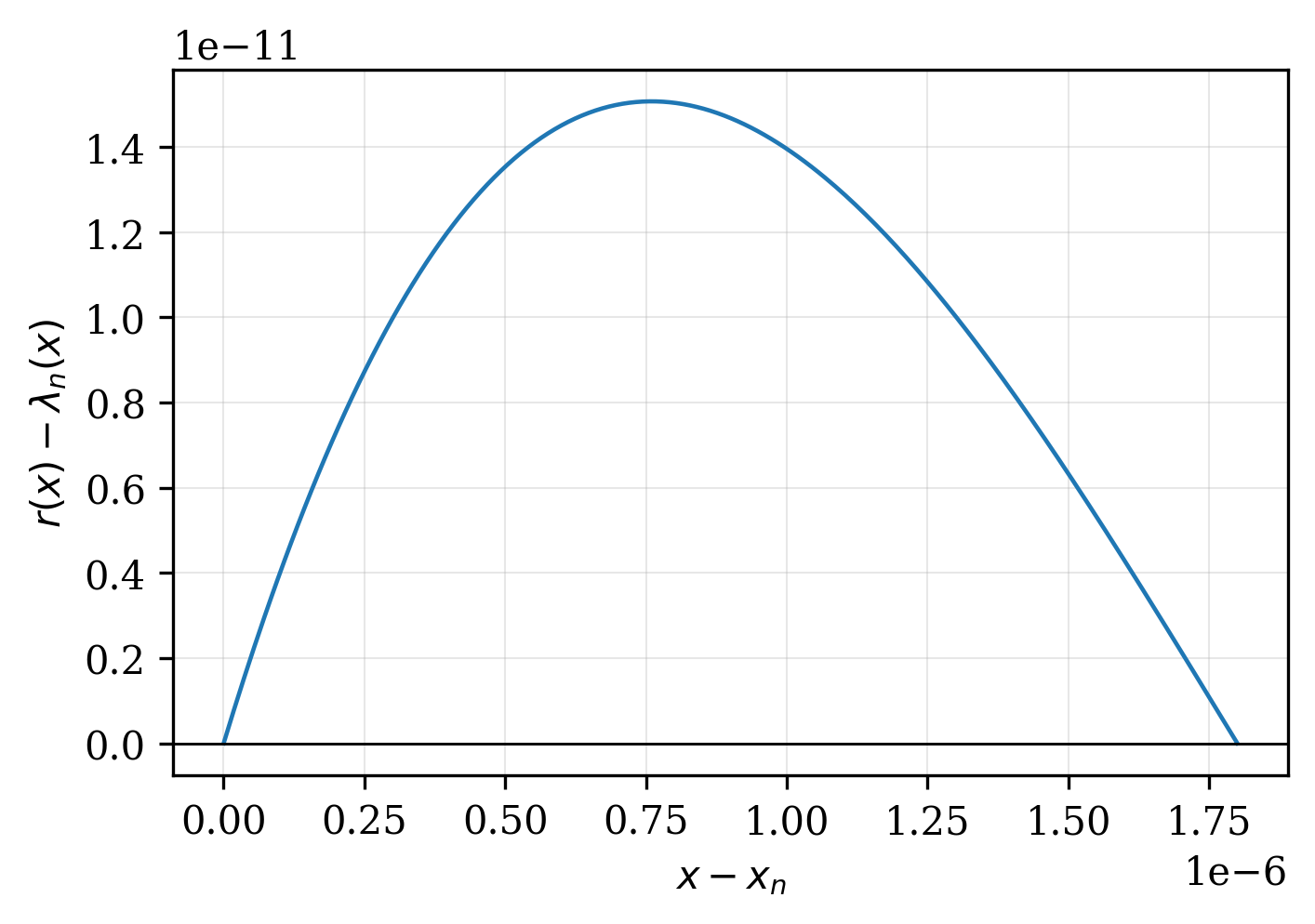}
    \caption{The difference $r(x)-\lambda_n(x)$ near $x_n$, where $r$ is the affine function joining the two selected points on the graph of $\lambda_n$.}
    \end{subfigure}
    \caption{Convexity of the Lebesgue function for the Chebyshev points of the first kind for degree 36.}
\end{figure}
\begin{table*}[ht]
\centering
\begin{tabular}{|c|c|} \hline
Index $m$ & Min. degree \\ \hline
1  &  36 \\ \hline
2  &  230 \\ \hline
3  &  1287  \\ \hline
\end{tabular}
\captionsetup{width=\textwidth}
\caption{Minimal degree $n$, numerically detected, ensuring local convexity near $x_{k+1}$ on $(x_k,x_{k+1})$ for every $k=1,\dots,m$ for Chebyshev points of the first kind for $m=1,2,3$.}\label{Table_convexity_CheI}
\end{table*} \label{table:Cheb1}

\noindent The convexity results in~\ref{Table_convexity_CheII} and~\ref{Table_convexity_CheI}  were obtained by computing the second derivatives of the polynomial pieces of the Lebesgue function at the relevant interval endpoint. By continuity, positivity at an endpoint implies positivity in a sufficiently small neighbourhood within the corresponding interval. The computations were carried out, using the \textit{Chebfun} package for MATLAB \cite{chebfun} and the Python library \textit{mpmath} \cite{mpmath}, which supports arbitrary-precision floating-point arithmetic.
For the latter, the working precision was initially set to 30 decimal digits, with computations repeated at 60 or more digits to check the stability of the computed values and their signs. The Python code used for the numerical experiments throughout this paper is available on \href{https://github.com/Mateusz-Suder-JU/Geometric-properties-of-the-Lebesgue-function}{GitHub}.

\section{Two-dimensional case}

\subsection{Padua points}
One of the most extensively studied sets of interpolation nodes in the bivariate setting is the Padua points on the square, introduced in \cite{marchi5}. They are the only known set of multivariate interpolation points for which the Lebesgue constant achieves the optimal lower bound for the rate of increase, namely $O(\log^2(n))$. 

\pagebreak

\noindent Let us denote Padua points as \vspace{-2mm}
\begin{equation}\label{eq:Padua_points}
\operatorname{Pad}_n
= \left\{ \left(\cos\frac{j\pi}{n},\,\cos\frac{k\pi}{n+1}\right) : 0\le j\le n,\quad 0\le k\le n+1,\quad j+k\text{ odd} \right\}.
\end{equation}

\noindent The set of Padua points can also be introduced in several other ways, for instance through the self-intersections of certain Lissajous curves, as discussed in \cite{PD1}.

\subsection{Morrow-Patterson points}

Another set of interpolation nodes of interest on the square is the so-called Morrow-Patterson points. \vspace{-2mm}

\begin{equation}\label{eq:Morrow-Patterson_points}
\operatorname{MP}_n
= \left\{\left(\cos\frac{j\pi}{n+2},\,\cos\frac{k\pi}{n+3}\right) : 1\le j\le n+1,\; 1\le k\le n+2,\;\; j+k\text{ odd} \right\}.
\end{equation}

\noindent Comparing \ref{eq:Padua_points} and \ref{eq:Morrow-Patterson_points}, we see that
\begin{equation*}
    MP_{n} = Pad_{n+2} - \text{edge points.}
\end{equation*}
We use this definition for both even and odd degrees $n$. After shifting both indices in \ref{eq:Morrow-Patterson_points} by one, we obtain the same index ranges and the same parity condition as in \ref{eq:Padua_points}. Moreover, in each coordinate the values with even and odd indices strictly interlace. Thus $\mathrm{MP}_n$ has the same arrangement of two interlacing rectangular grids as $\mathrm{Pad}_n$, and~\cite{Floater2017} implies that $\mathrm{MP}_n$ is unisolvent for $\mathcal{P}_n(\mathbb{R}^2)$ for every $n\ge 1$.
The Morrow-Patterson points are sub-optimal compared to Padua points in the sense that the Lebesgue constants of $MP_n$ points grow asymptotically like $O(n^2)$ \cite{MP1}, \cite{MP2}.

\subsection{Local maxima of the Lebesgue function on the square}

As mentioned in Subsection 2.1, in the case of the interval, by considering the zeros of the derivative of the Lebesgue function restricted to $(x_k, x_{k+1})$, one can show that $\lambda_n$ attains its maximum on $[x_k, x_{k+1}]$ at exactly one point. Consequently, the Lebesgue function of $n+1$ nodes has at most $n$ local maxima on the interval $[x_1,x_{n+1}]$.

\vskip 2mm
\noindent We are interested in whether a similar estimate for the number of local maxima can be extended to the Lebesgue function on the square. Throughout, local maxima on the boundary of the square are understood
relative to the square, rather than to the boundary itself.

\vskip 2mm
\noindent Our numerical experiments suggest the following conjecture.

\vskip 4mm
\begin{conjecture}
For both the Padua and MP points, the number of local maxima
of the Lebesgue function in the interior of the square
and on the entire square is at least
\[
\frac{n(n-1)}{2}
\qquad \text{and} \qquad
\frac{(n+1)(n+2)}{2},
\]
respectively. These bounds are not always attained, as illustrated
in Tables \hyperref[number_of_loc_max_Pad]{3} and \hyperref[number_of_loc_max_MP]{4}.
\end{conjecture}

\vskip 5mm

\Needspace{10\baselineskip}
\noindent
\begin{minipage}{0.45\textwidth}
\textbf{Table 3.} Number of numerically
detected local maxima of Lebesgue function for Padua points.
\end{minipage}
\hspace{0.075\textwidth}
\begin{minipage}{0.45\textwidth}
\textbf{Table 4.} Number of numerically
detected local maxima of Lebesgue function for MP points.
\end{minipage}
\vspace{-4mm}
\begin{table}[h]
\begin{minipage}[c]{0.5\textwidth}
\begin{tabular}{|c|c|c|}
\hline
deg. $n$ & \shortstack{number of inner\\local maxima} & \shortstack{number of \\ local maxima} \\
\hline
3 & 7 & 15 \\ \hline
4 & 13 & 27 \\ \hline
5 & 14 & 25 \\ \hline
6 & 23 & 39 \\ \hline
7 & 33 & 50 \\ \hline
8 & 44 & 62 \\ \hline
\end{tabular}
\label{number_of_loc_max_Pad}
\end{minipage}
\hspace{1cm}
\begin{minipage}[c]{0.5\textwidth}
\begin{tabular}{|c|c|c|}
\hline
deg. $n$ & \shortstack{number of inner\\local maxima} & \shortstack{number of\\ local maxima} \\
\hline
3 & 9 & 16 \\ \hline
4 & 8 & 17 \\ \hline
5 & 27 & 38 \\ \hline
6 & 19 & 32 \\ \hline
7 & 27 & 42 \\ \hline
8 & 36 & 53 \\ \hline
\end{tabular}
\label{number_of_loc_max_MP}
\end{minipage}
\end{table}

\vskip 4mm 
\noindent Let $\mathcal A$ be a set of interpolation nodes on a square. Each fundamental polynomial $\ell_{A^0}^{\mathcal A}$ has a constant sign on every region separated by its zero set, see e.g. Fig.~\ref{Fundamental_polynomial_padua}. For a fixed node $A^0\in\mathcal A$, write
\[
\lambda_n^{\mathcal A}(x,y) = g(x,y)+|\ell_{A^0}^{\mathcal A}(x,y)|, 
\qquad g(x,y) = \sum_{A\in\mathcal A,\,A\ne A^0} |\ell_A^{\mathcal A}(x,y)|.
\]
At an interior regular point of the curve $\{\ell_{A^0}^{\mathcal A}=0\}$, the term $|\ell_{A^0}^{\mathcal A}|$ grows approximately linearly in either normal direction. The Lebesgue function thus coincides with $g$ on the zero curve and exceeds it elsewhere, which may alter the shape of a local peak of $g$ near the curve. The numerical examples for Padua and Morrow-Patterson points are consistent with this interpretation: the addition of $|\ell_{A^0}^{\mathcal A}|$ may reshape a local peak of $g$ near the zero curve into a region with several nearby maxima. However, neither proximity to the zero curve nor the transverse cusp alone determines the existence or number of neighbouring maxima, see Figs.~\ref{Multiple_maximums_plots_Padua_points} and~\ref{Multiple_maximums_plots_MP_points}.

\begin{figure}[htbp]
    \includegraphics[width=\textwidth]{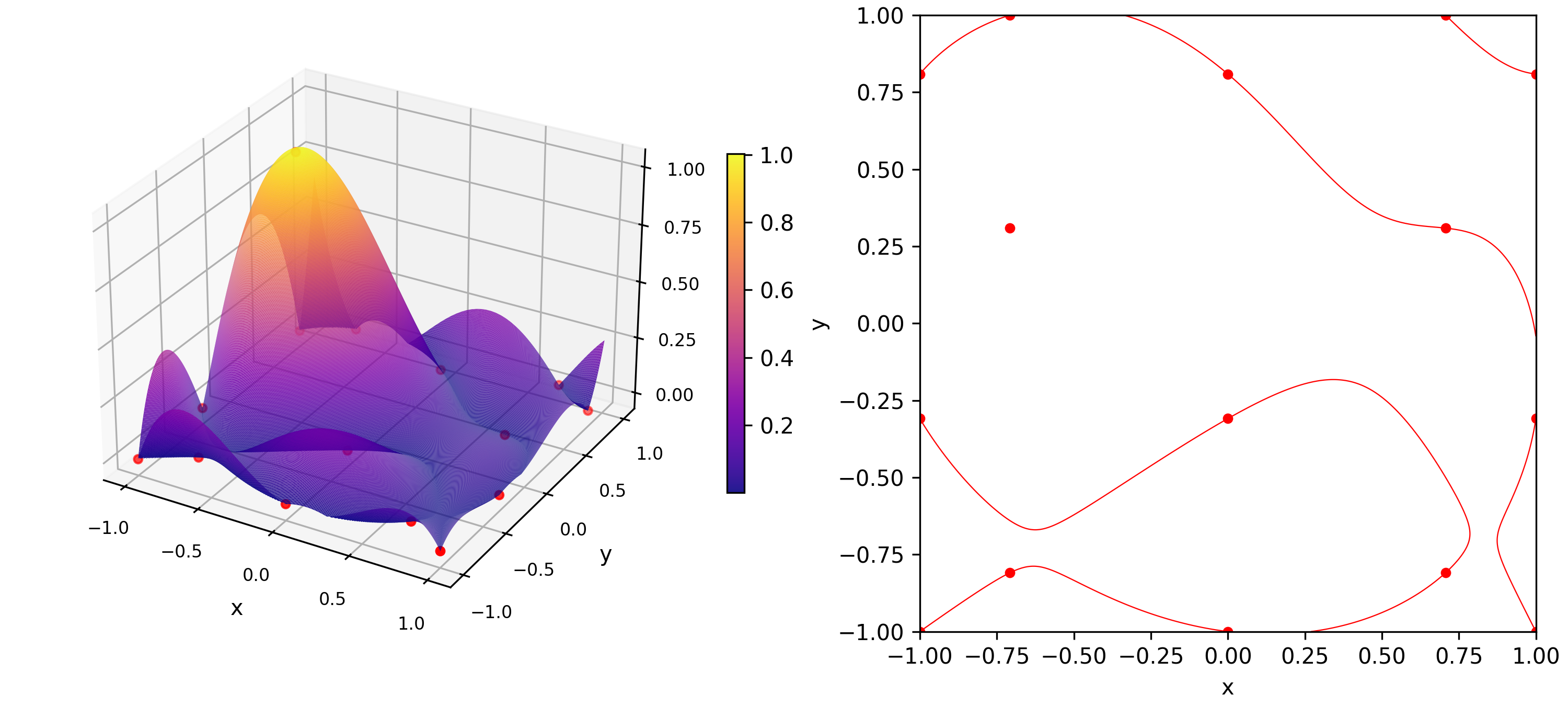}
    \caption{The fundamental Lagrange polynomial associated with the Padua point $A=(\cos(3\pi/4),\cos(2\pi/5))$ for degree $4$. Left: its absolute value, with red markers showing the absolute values at the interpolation nodes. Right: its zero set, with the interpolation nodes marked in red.}
    \label{Fundamental_polynomial_padua}
\end{figure}

\noindent The local maxima shown in Figs.~\ref{Multiple_maximums_plots_Padua_points} and~\ref{Multiple_maximums_plots_MP_points} are found numerically using the following procedure. Starting from points distributed over the square, numerical maximization is used to locate candidate local maxima. Each candidate is examined in a small rectangular neighbourhood by comparing a value of the Lebesgue function attained inside it with upper bounds for its values on the boundary. Concavity and monotonicity tests are then used to check that the neighbourhood contains only one local maximum. If uniqueness is confirmed, the rectangle is excluded from further search.

\vspace{2mm}
\noindent
Next, the square is subdivided into smaller rectangles. For a polynomial $p$ with Bernstein coefficients $b_i$ on a rectangle $B$, the inequality
\[
\min_i b_i \leq p(x,y) \leq \max_i b_i,
\qquad (x,y)\in B,
\]
provides bounds for its values throughout $B$. Applying this inequality to the fundamental polynomials and their derivatives gives bounds used to identify rectangles on which the Lebesgue function is strictly increasing or decreasing in one coordinate direction. Such rectangles are excluded whenever the direction of increase remains within the square.

\vspace{2mm}
\noindent
The remaining rectangles are subdivided further and used to provide additional starting points for another application of the numerical procedure described above. For Figs.~\ref{Multiple_maximums_plots_Padua_points} and~\ref{Multiple_maximums_plots_MP_points}, the maximum subdivision depth was set to $11$. Rectangles that remained unresolved at this depth are marked in light blue. All computations were performed in double precision. The code used for the numerical experiments is available on \href{https://github.com/Mateusz-Suder-JU/Geometric-properties-of-the-Lebesgue-function}{GitHub}.

\begin{figure}[H]
\centering
     \begin{subfigure}[t]{0.4\textwidth}
        \includegraphics[width=\textwidth]{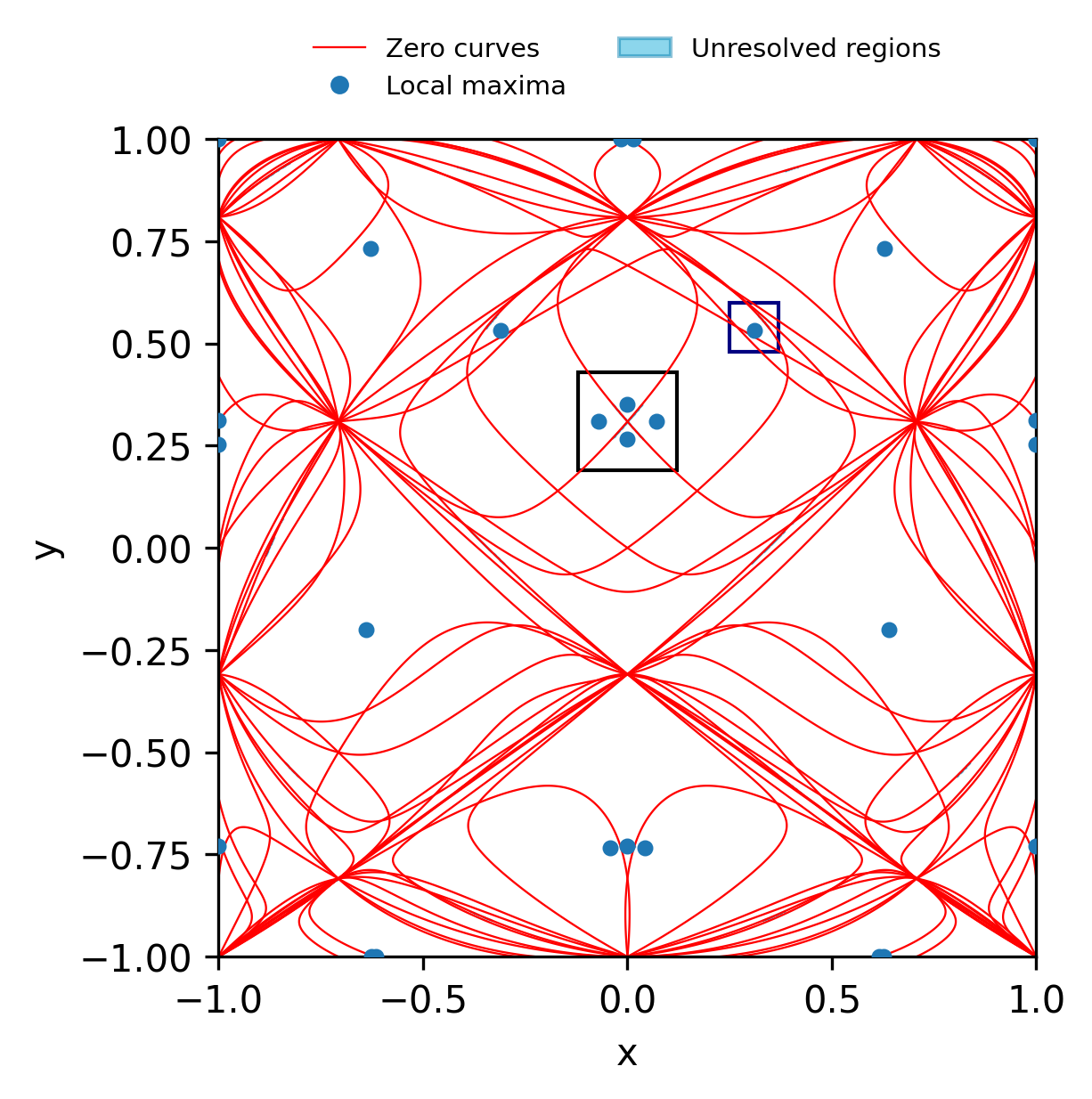}
        \caption{Zero sets of the fundamental Lagrange polynomials for the Padua points of degree $4$.}
    \end{subfigure}
    \hspace{0.1\textwidth}
    \begin{subfigure}[t]{0.45\textwidth}
        \includegraphics[width=\textwidth]{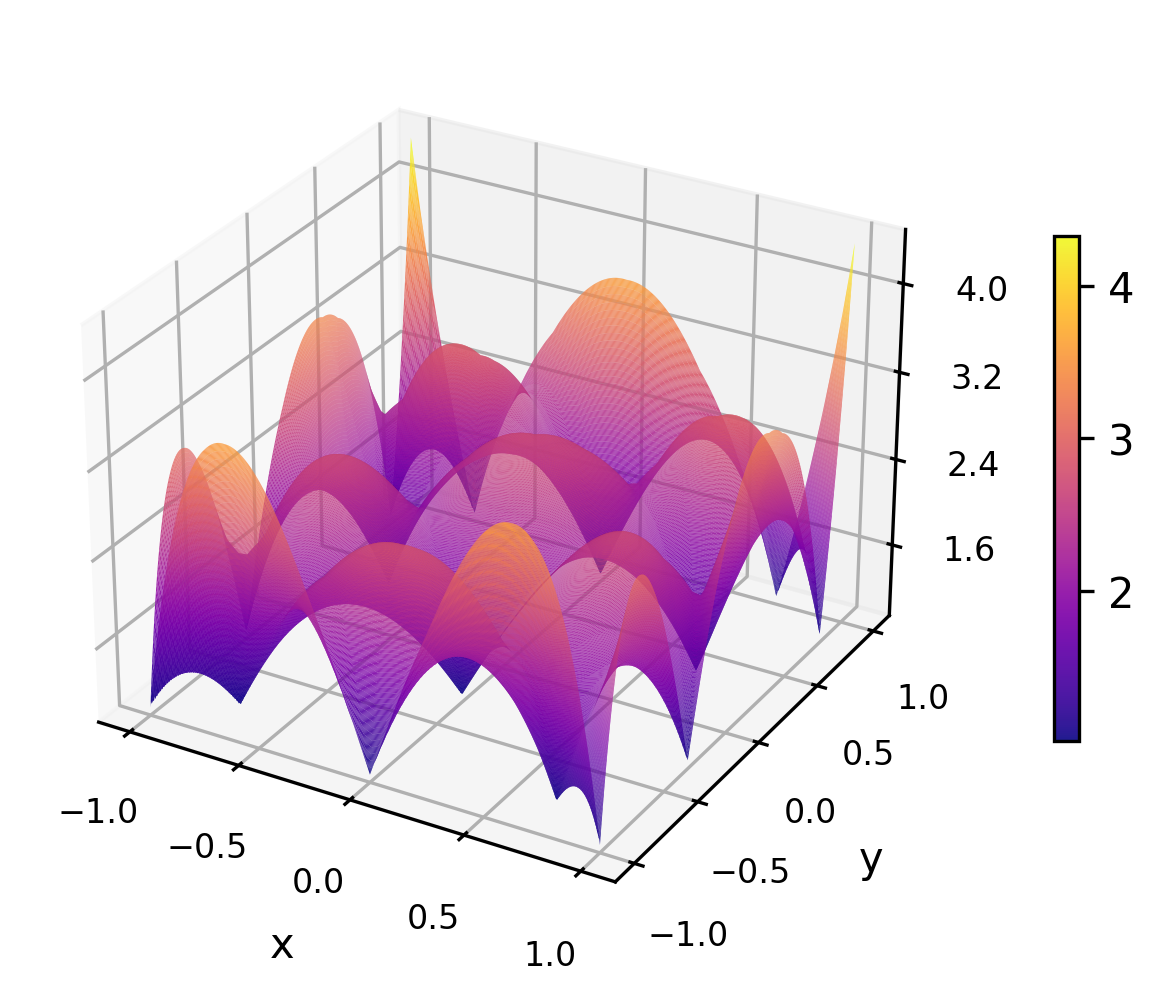}
        \caption{The Lebesgue function for the Padua points of degree $4$.}
    \end{subfigure}

     \begin{subfigure}[t]{0.4\textwidth}
        \includegraphics[width=\textwidth]{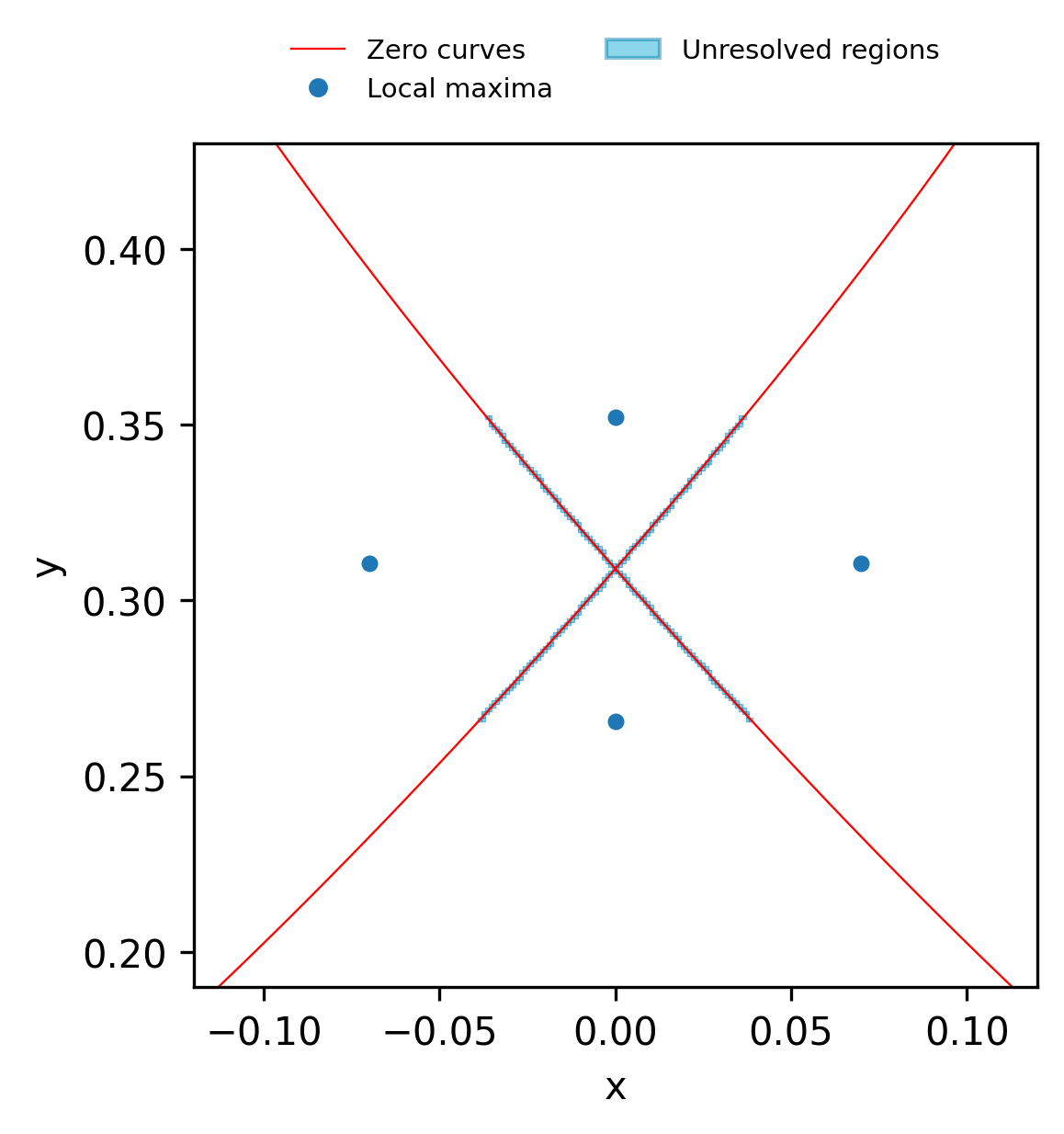}
        \caption{Enlargement of the region outlined in black in panel~(a).}
    \end{subfigure}
    \hspace{0.1\textwidth}
    \begin{subfigure}[t]{0.45\textwidth}
        \includegraphics[width=\textwidth]{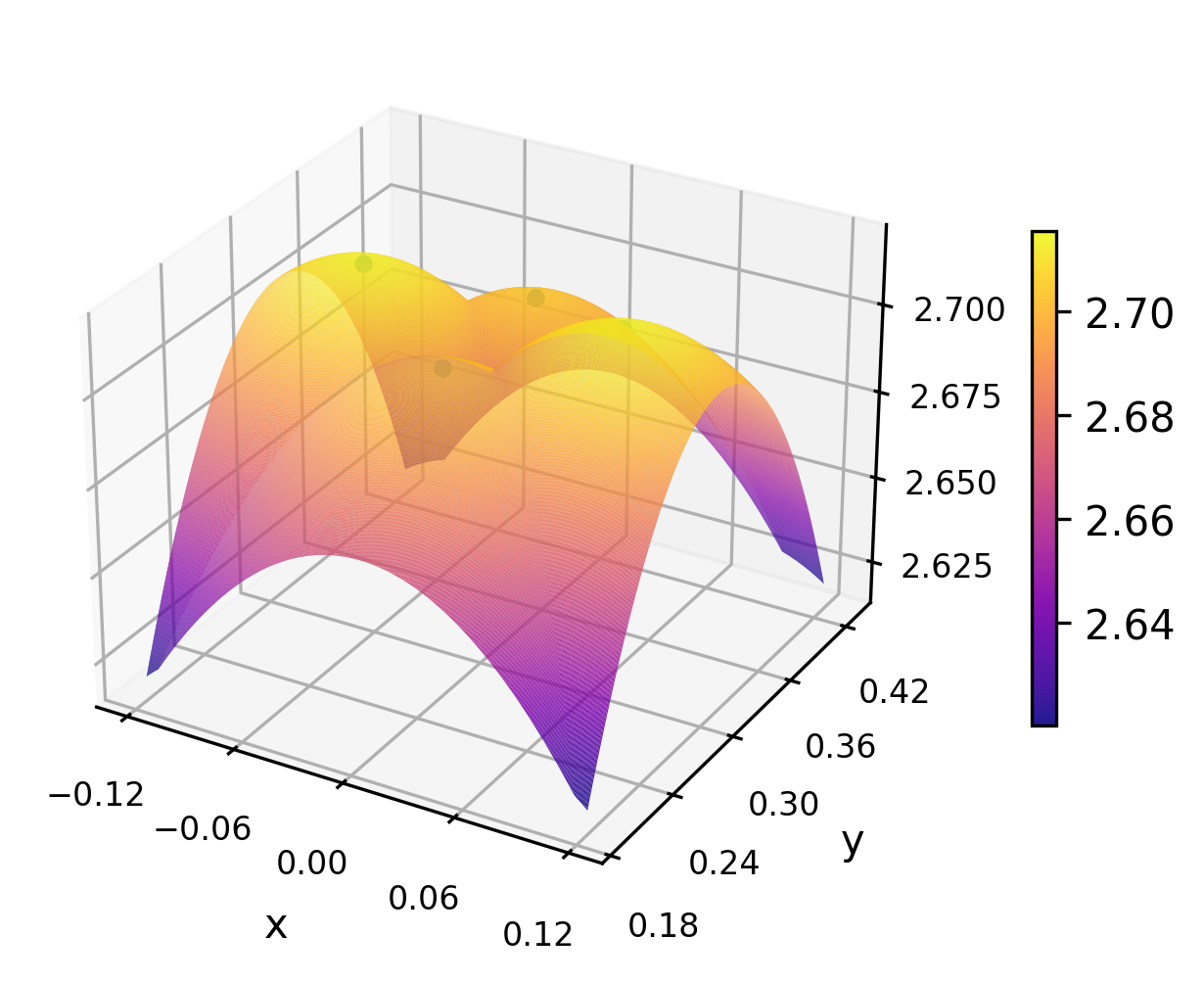}
        \caption{The surface of the Lebesgue function over the region shown in panel~(c).}
    \end{subfigure} 
    
    \begin{subfigure}[t]{0.4\textwidth}
        \includegraphics[width=\textwidth]{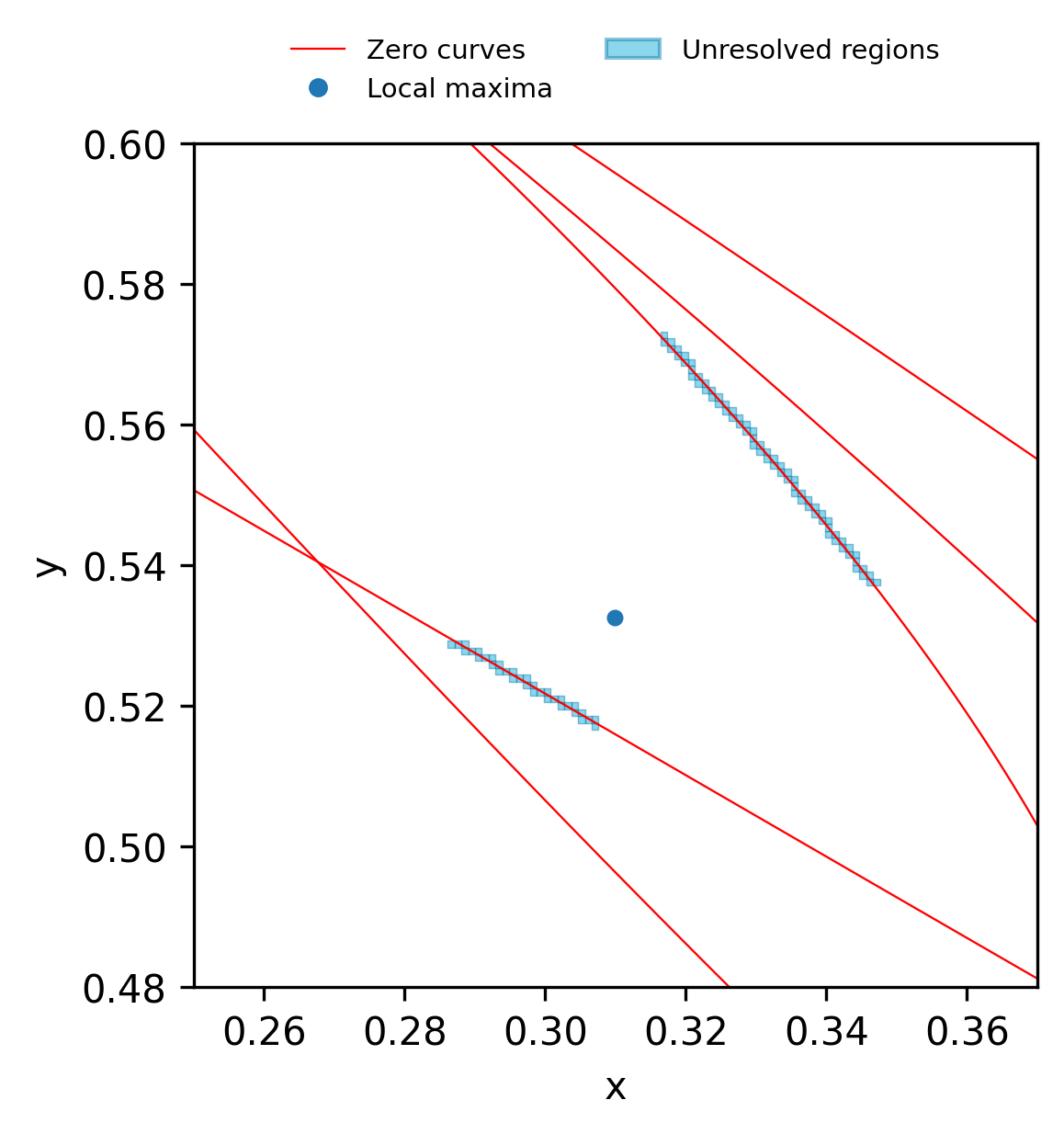}
        \caption{Enlargement of the region outlined in dark blue in panel~(a).}
    \end{subfigure}
    \hspace{0.1\textwidth}
    \begin{subfigure}[t]{0.45\textwidth}
        \includegraphics[width=\textwidth]{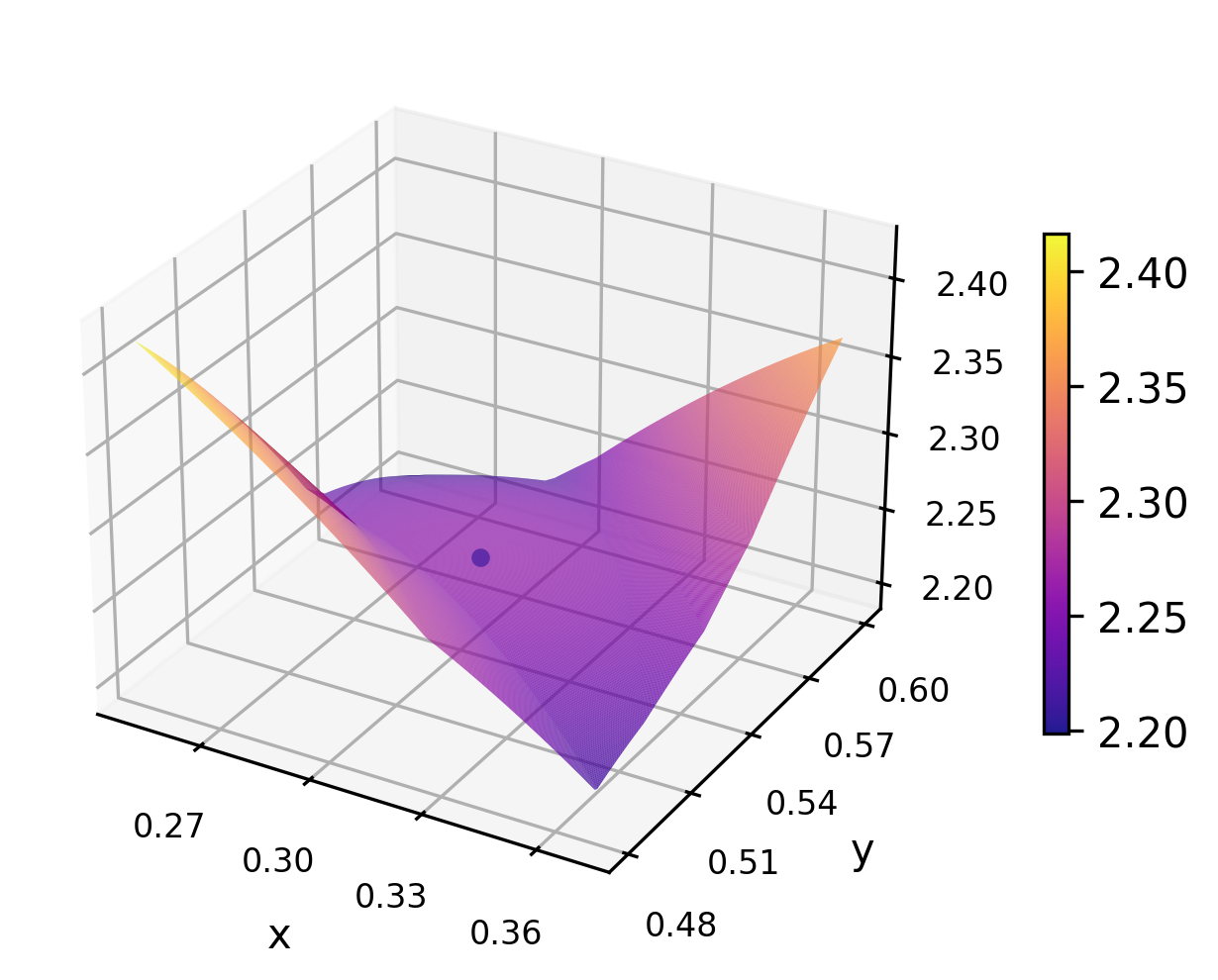}
        \caption{The surface of the Lebesgue function over the region shown in panel~(e).}
    \end{subfigure} 
    \caption{Lebesgue function and its local maxima for the Padua points of degree 4.}
    \label{Multiple_maximums_plots_Padua_points}
\end{figure}

\begin{figure}[H]
\centering
     \begin{subfigure}[t]{0.4\textwidth}
        \includegraphics[width=\textwidth]{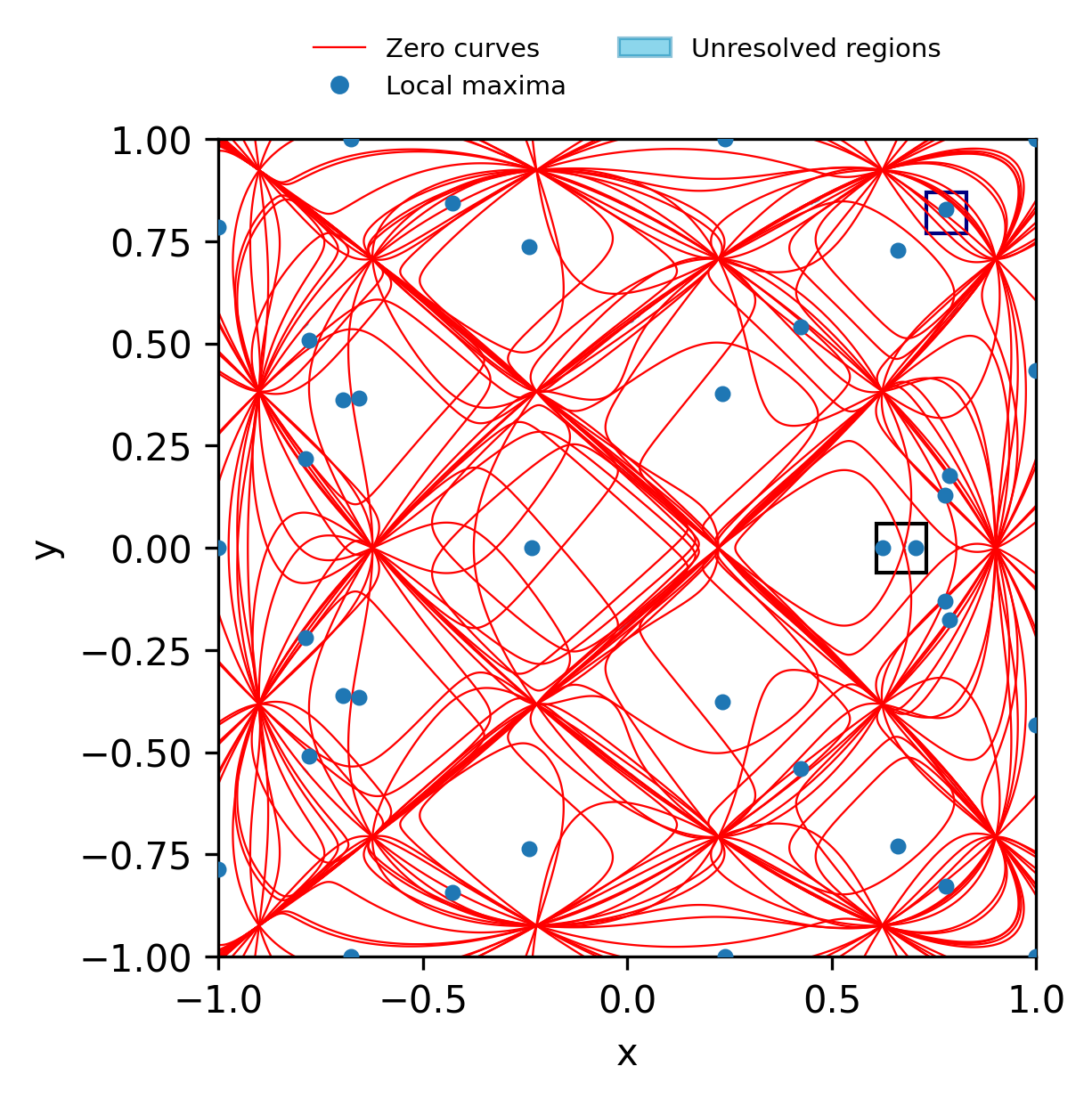}
        \caption{Zero sets of the fundamental Lagrange polynomials for the Morrow-Patterson points of degree $5$.}
    \end{subfigure}
    \hspace{0.1\textwidth}
    \begin{subfigure}[t]{0.45\textwidth}
        \includegraphics[width=\textwidth]{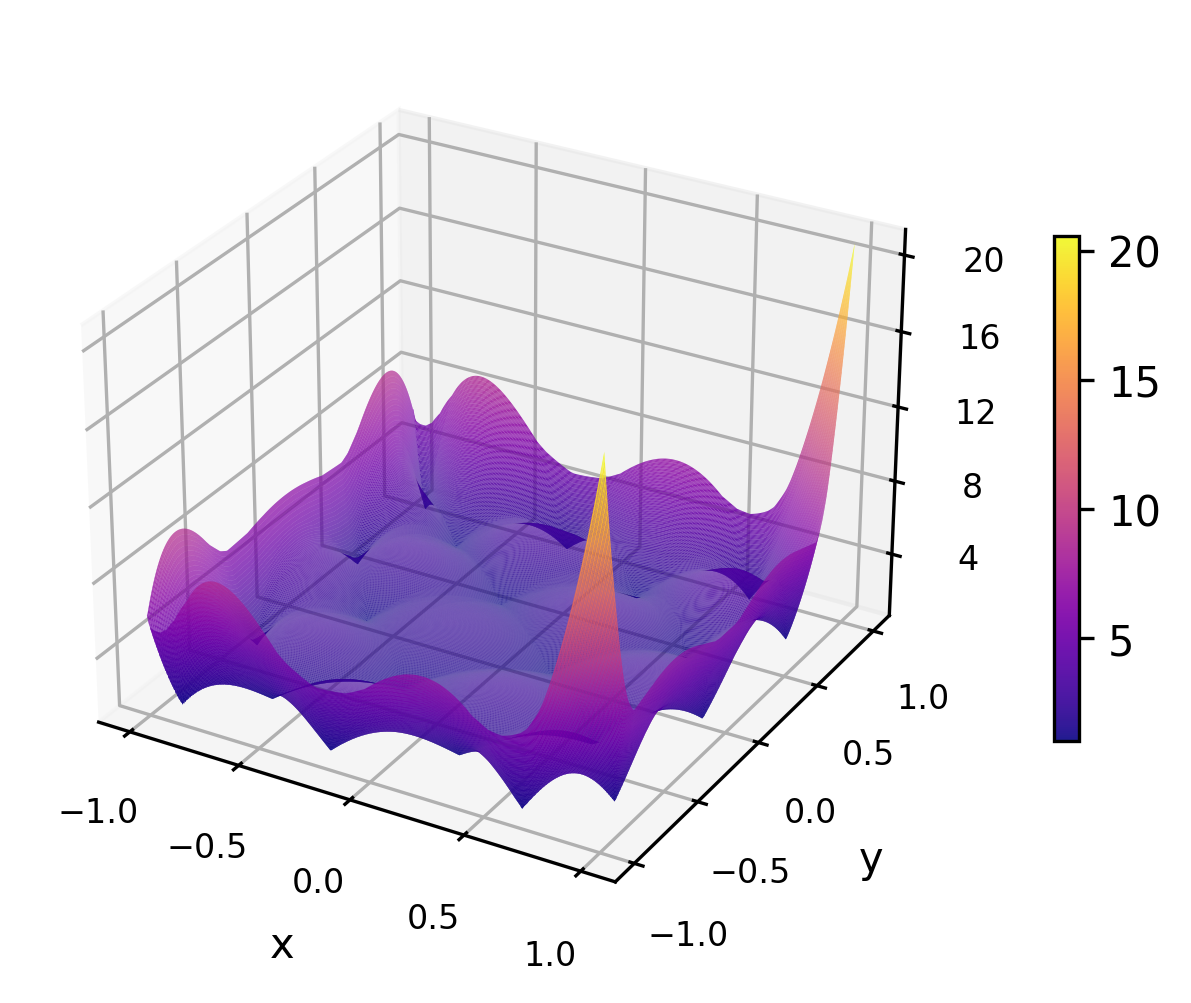}
        \caption{Lebesgue function of the MP points of degree $5$.}
    \end{subfigure}
    
    \begin{subfigure}[t]{0.4\textwidth}
        \includegraphics[width=\textwidth]{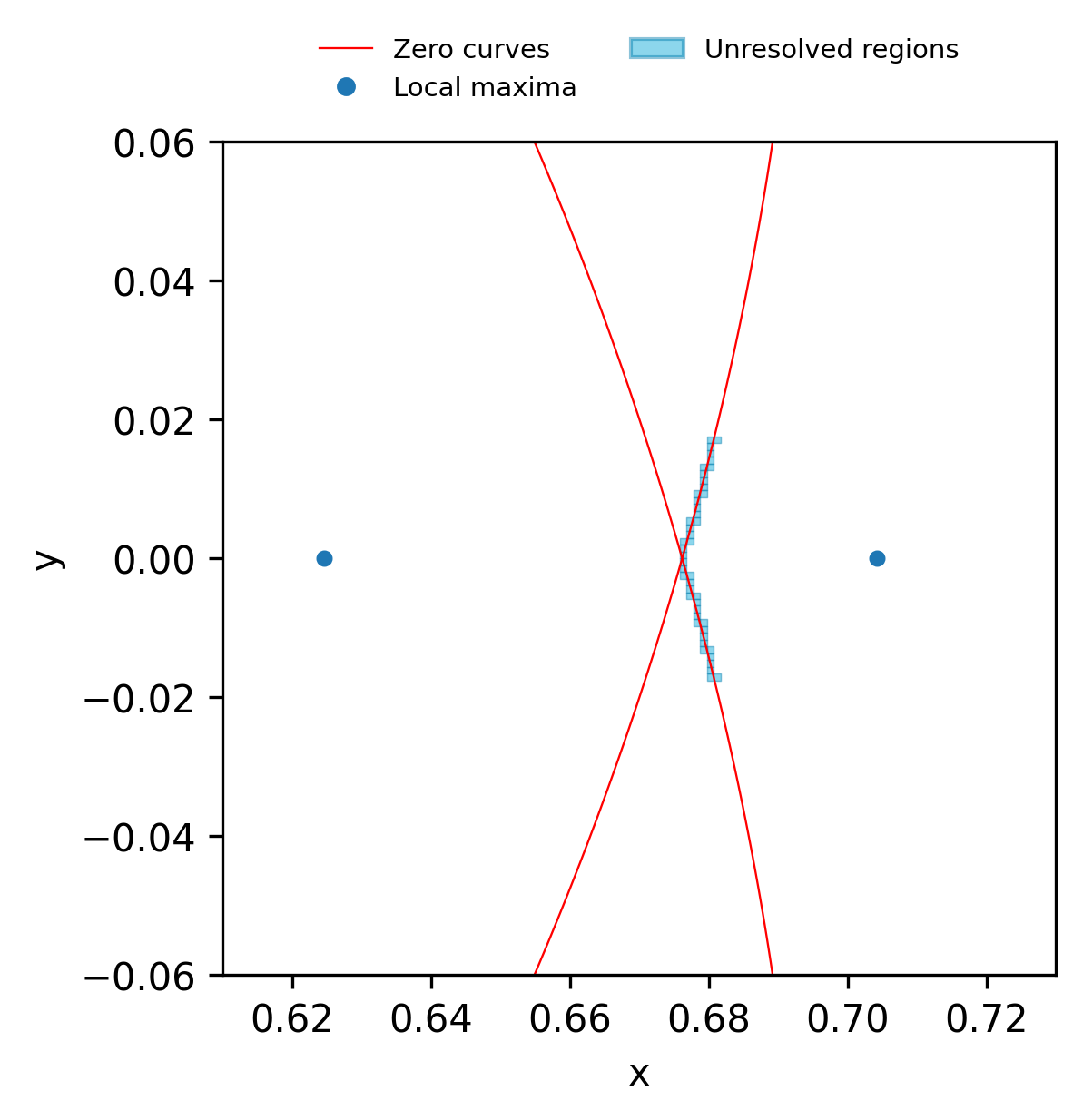}
        \caption{Enlargement of the region outlined in black in panel~(a).}
    \end{subfigure}
    \hspace{0.1\textwidth}
    \begin{subfigure}[t]{0.45\textwidth}
        \includegraphics[width=\textwidth]{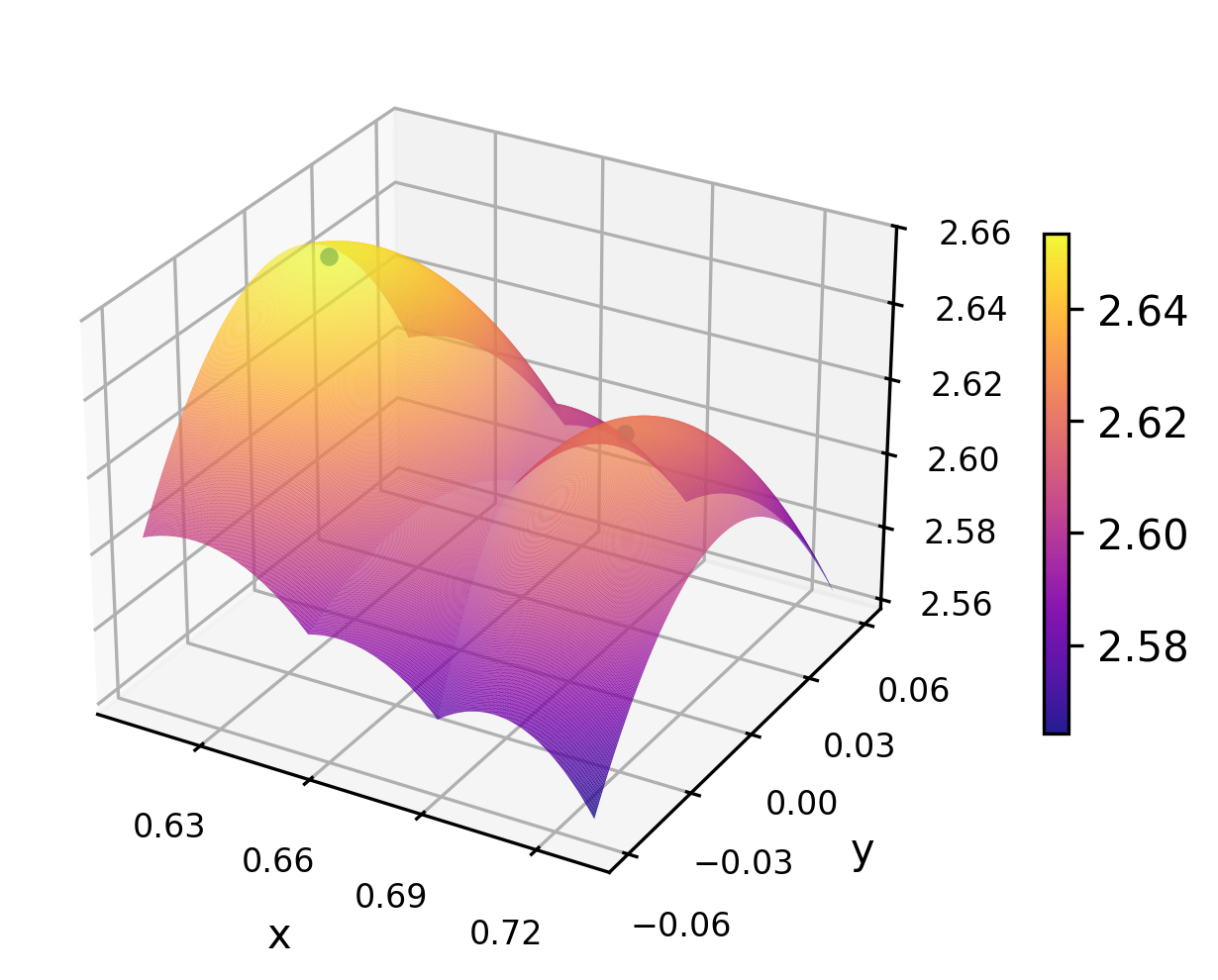}
        \caption{The surface of the Lebesgue function over the region shown in panel~(c).}
    \end{subfigure} 

    \begin{subfigure}[t]{0.4\textwidth}
        \includegraphics[width=\textwidth]{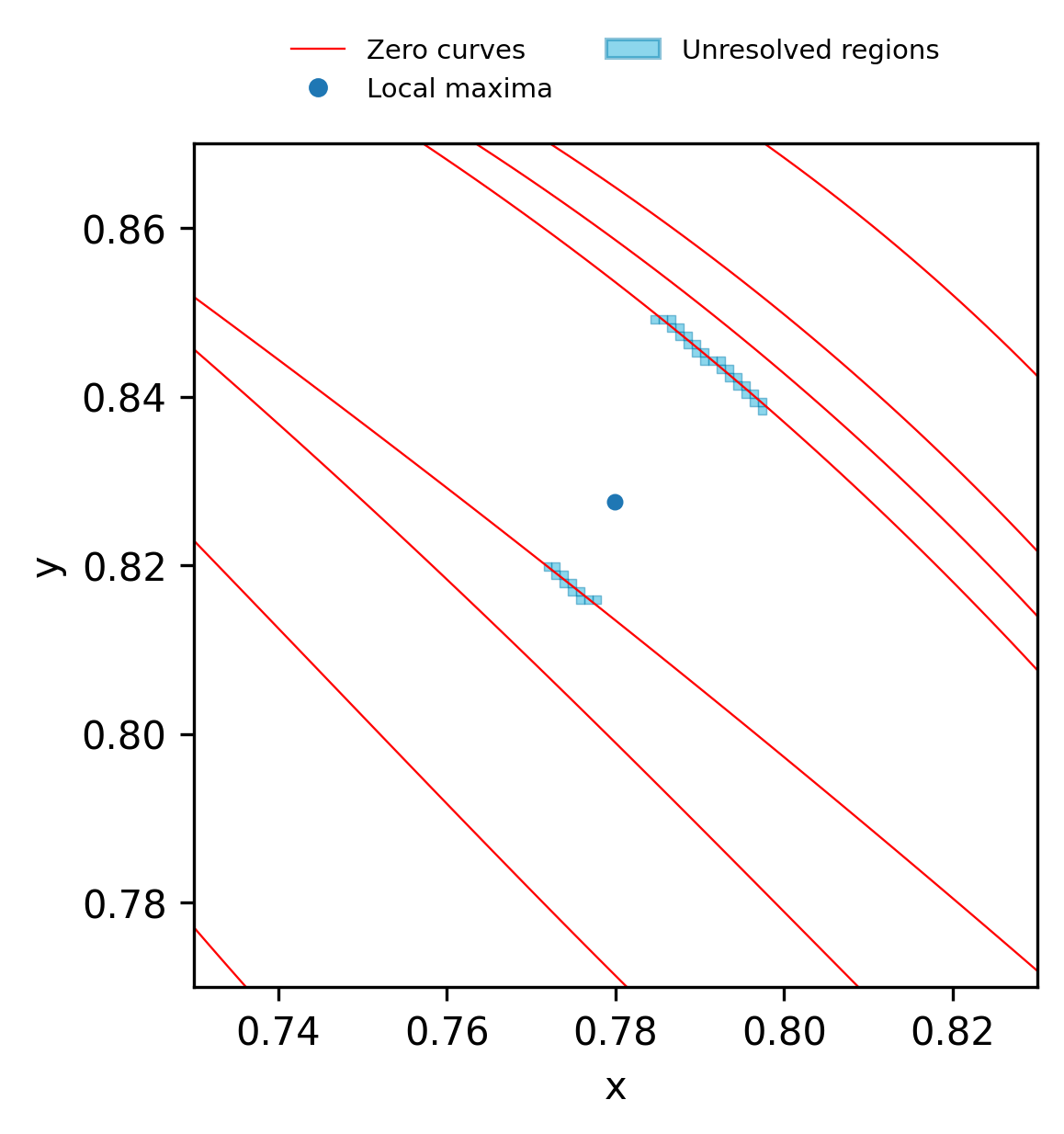}
        \caption{Enlargement of the region outlined in dark blue in panel~(a).}
    \end{subfigure}
    \hspace{0.1\textwidth}
    \begin{subfigure}[t]{0.45\textwidth}
        \includegraphics[width=\textwidth]{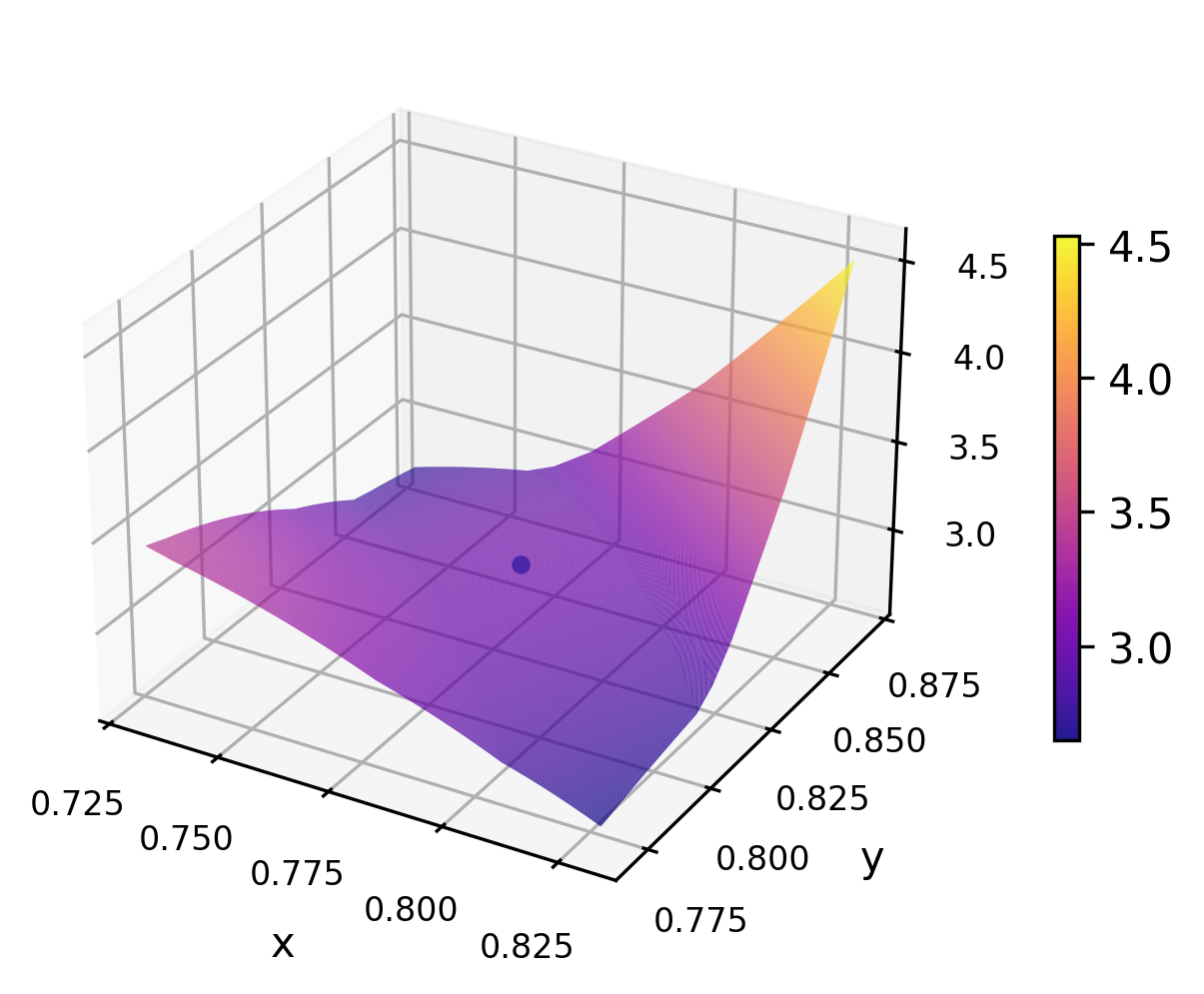}
        \caption{The surface of the Lebesgue function over the region shown in panel~(e).}
    \end{subfigure}
    \caption{Lebesgue function and its local maxima for the Morrow-Patterson points of degree $5$.}
    \label{Multiple_maximums_plots_MP_points}
\end{figure}

\section{Concluding remarks and open problems}
\noindent As illustrated in Fig. \ref{Cheb_small_perturbation}, the location of the global maximum of the Lebesgue function in the interval can change even with minute changes of the interpolation nodes. We suspect that a condition on the node distribution over the entire interval, based on the covering radius with respect to an arccos-type metric, may provide further insight into the overall geometry of the Lebesgue function. 

\vskip 2mm
\noindent
Proposition and Theorem proved in Subsection 2.1  provide a necessary condition for the Lebesgue function associated with a general family of nodes to attain its maximum at the point~1. A natural question is which sufficient conditions guarantee this property.

\vskip 2mm
\noindent
The degrees of the Lebesgue function for which it is convex on certain subintervals, discussed in Subsection 2.2, appear to become increasingly sparse for point sets in which the maximal difference between local maxima decreases. In particular, as point sets approach optimal interpolation nodes, the regions of convexity seem to diminish in size and occur less frequently. It remains an open question whether the Lebesgue function of optimal interpolation points on an interval contains any regions of convexity between interpolation points. Verifying such a property would be an interesting problem for the characterization of optimal node sets. However, due to the high computational complexity of approximating those sets with high accuracy at large degrees, verifying this property numerically might prove challenging.

\vskip 2mm 
\noindent
In the two-dimensional case, for both the Padua and MP points the number of numerically detected local maxima for $3\leq n\leq 8$ exceeds the conjectured lower bound as illustrated in Figs.~\ref{Multiple_maximums_plots_Padua_points} and \ref{Multiple_maximums_plots_MP_points}. Estimating the number of these additional maxima, as well as explicitly characterizing the conditions under which multiple nearby local maxima occur, would provide valuable insight into the general geometry of the Lebesgue function in the bivariate case.

\vskip 4mm

{\bf Acknowledgments.}\\
The authors would like to thank the reviewer for suggesting a refinement of the argument in Prop.$\,$1, which led to an improvement of the constant in the corresponding estimates.

\vskip 2mm
The collaboration of the authors in Padua, as well as the research stay of the third author in Kraków, were supported by the program Excellence Initiative at the Jagiellonian University in Krakow (ID.UJ).
The third author thanks the INdAM-GNCS group, the thematic group on "Approximation Theory and Applications" of the Italian Mathematical Union.

\end{document}